\documentclass[a4paper,12pt]{article}
\usepackage{CJKutf8}
\usepackage{CJK}
\usepackage{amsfonts}
\usepackage{pifont}
\usepackage{bm}
\usepackage{latexsym,amsmath,amssymb,cite,amsthm}
\usepackage{color,eucal,enumerate,mathrsfs}
\usepackage[normalem]{ulem}
\usepackage{amsmath}
\usepackage[colorlinks, linkcolor=magenta,  anchorcolor=blue,citecolor=blue]{hyperref}

\numberwithin{equation}{section}
\theoremstyle{plain}
\newtheorem{theorem}{Theorem}[section]
\newtheorem{corollary}[theorem]{Corollary}
\newtheorem{lemma}[theorem]{Lemma}
\newtheorem{proposition}[theorem]{Proposition}
\theoremstyle{definition}
\newtheorem{definition}[theorem]{Definition}
\theoremstyle{remark}

\newcommand{\geo}{\rm Geo}

\newcommand{\lmt}[2]{\mathop{\lim}_{{#1} \rightarrow {#2}} }

\newcommand{\lip}[1]{{\mathrm{lip}}({#1})}

\newcommand{\lmts}[2]{\mathop{\overline{\lim}}_{{#1} \rightarrow {#2}} }
\newcommand{\lmti}[2]{\mathop{\underline{\lim}}_{{#1} \rightarrow {#2}} }

\newcommand{\Ric}{{\rm{Ricci}}}

\newcommand{\bRic}{{\bf Ricci}}

\renewcommand{\H}{{\mathrm{Hess}}}

\newcommand{\Hess}{{\mathrm{Hess}}}

\newcommand{\mm}{\mathfrak m}
\newcommand{\ms}{(X,\d,\mm)}
\newcommand{\cd}{{\rm CD}(K, \infty)}
\newcommand{\cdkn}{{\rm CD}(K, N)}
\newcommand{\mcpkn}{{\rm MCP}(K, N)}
\newcommand{\rcdkn}{{\rm RCD}(K, N)}
\newcommand{\rcd}{{\rm RCD}(K, \infty)}
\newcommand{\be}{{\rm BE}(K, \infty)}
\newcommand{\bekn}{{\rm BE}(K, N)}
\newcommand{{\vol}}{{\rm Vol}}
\newcommand{\g}{{\rm g}}
\newcommand{\La}{\mathrm{L}}

\newcommand{\Q}{\mathfrak{Q}}
\newcommand{\R}{\mathbb{R}}

\newcommand{\supp}{\mathop{\rm supp}\nolimits}   
\newcommand{\Lip}{\mathop{\rm Lip}\nolimits}
\newcommand{\loc}{{\rm loc}}
\renewcommand{\d}{{\mathrm d}}

\newcommand{\D}{{\mathrm D}}
\newcommand{\restr}[1]{\lower3pt\hbox{$|_{#1}$}}
\newcommand{\la}{{\langle}}
\newcommand{\ra}{{\rangle}}

\newcommand{\nchi}{{\raise.3ex\hbox{$\chi$}}}

\title{\large{\bf Measure rigidity of synthetic lower Ricci curvature bound on Riemannian manifolds}
}

\begin{document}
\author{Bang-Xian Han\thanks{Department of Mathematics, Technion-Israel Institute of Technology, han@technion.ac.il 
}
}

\date{\today} 
\maketitle

\begin{abstract}
In this paper we investigate  Lott-Sturm-Villani's synthetic lower Ricci curvature bound on Riemannian manifolds  with boundary.  We prove several measure rigidity results related to  optimal transport on Riemannian manifolds, which  completely characterize   $\cd$ condition and non-collapsed $\cdkn$ condition on Riemannian manifolds with boundary. 
In particular, we reveal the measure rigidity  of  Riemannian interpolation inequality proved by Cordero-Erausquin, McCann and Schmuckenschl\"ager.
We prove  that  log-(semi)concave  measures  are  the only reference measures so that displacement convexity holds on Riemannian manifolds.  This is the first measure rigidity result concerning the synthetic dimension-free  CD  condition, which is new even on $\R^n$.     Using   $L^1$-optimal transportation  theory,   we  also prove that CD condition  yields  geodesical convexity (with respect to the usual Riemannian distance).

\end{abstract}

\tableofcontents

\section{Introduction}

The synthetic theory of metric measure spaces with lower Ricci curvature bounds, initiated by Lott-Villani \cite{Lott-Villani09} and Sturm \cite{S-O1, S-O2}, has remarkable developments in recent years. 
We refer the reader to the survey \cite{AmbrosioICM} for an overview of this topic  and  bibliography. 

Many important results, previously known on Riemannian manifolds with lower Ricci curvature bound, now have their generalized versions in  synthetic setting. However, we still do not fully understand synthetic lower Ricci curvature bound on Riemannian manifolds.  In particular, we know very little about the dimension-free $\cd$ condition.   

 In this paper we return to the starting point of this  rapidly developing theory,  investigate  the following  conjecture  using   new tools and results developed  in  recent years.
 
  \noindent
{ \bf Conjecture:}
{\it  Let $(M, \g)$ be a complete Riemannian manifold  and ${\Omega}$ be a bounded open set.  Denote by  $\d_\Omega$  the intrinsic distance on $\overline{\Omega}$ induced by  $\g$. Then $ (\overline{\Omega}, \d_\Omega, \mm)$  is $\cd$ in the sense of Lott-Sturm-Villani  {if and only if} the following conditions are satisfied:
\begin{itemize}
\item [1)]   $\mm=e^{-V}{\vol}$ for some semi-convex function $V$,
\item [2)] $\Omega$ is geodesically convex.
\end{itemize}  
}

\bigskip

Next, we  introduce some backgrounds and explain the motivation  in more detail.  Let $(M, {\rm g})$ be a $n$-dimensional Riemannian manifold, and $\mm:=e^{-V}{\rm Vol}_{\rm g}$ be a weighted measure for some smooth function $V$.     The  diffusion operator associated with
the  smooth metric measure space  $(M, {\rm g}, \mm)$  is $\La=\Delta-\nabla V$ where $\Delta$ is the Laplace-Beltrami operator, and the well-known  Bakry-\'Emery's  $\Gamma_2$ operator is  defined by
\[
\Gamma_2(f):=\frac12 \La \Gamma(f,f) -\Gamma(f, \La f),\qquad\Gamma(f,f):=\frac12 \La (f^2)-f\La f.
\]
  It is known that  $\Gamma(\cdot, \cdot)={\rm g}(\nabla \cdot, \nabla \cdot )$,  and we have the following  generalized  Bochner's   formula
\begin{equation}\label{bf}
\Gamma_2(f)=\Ric(\nabla f, \nabla f)+\H_V(\nabla f, \nabla f)+\| \H_f \|_{\rm HS}^2
\end{equation}
 for any  $f\in C^\infty(M)$, where $\H_V=\D^2 V$ is the Hessian of $V$ and $\| \H_f \|_{\rm HS}$ is the Hilbert-Schmidt norm of  $\H_f$. We say  that  $(M, {\rm g}, \mm)$ satisfies the  $(K, N)$-Bakry-\'Emery's  condition (or $\bekn$ condition for short),  if the following  generalized Bochner inequality holds
\begin{equation}\label{eq1-intro}
\Gamma_2(f) \geq K\Gamma(f)+\frac1N (\La f)^2,~~~~\forall  f\in C^\infty(M).
\end{equation}
  It is known that   $\bekn$ condition yields many important geometric and analytic properties.  For example,  when $N=\infty$, we have the following equivalent  characterizations, which are also regarded as generalized lower Ricci curvature bound (c.f. \cite{SVR-T}).
  \begin{itemize}
  \item [{\bf 0)}] {Modified Ricci tensor  bound}:
  $$
 \Ric_V(\nabla f, \nabla f) := \Ric(\nabla f, \nabla f)+\H_V(\nabla f, \nabla f) \geq K|\nabla f|^2
  $$ for all $f\in C^\infty(M)$.
\item [{\bf 1)}] {$\be$ condition}: $\Gamma_2(f) \geq K \Gamma(f)$  for all $f\in C^\infty(M)$.
\item [{\bf 2)}]  {$\cd$ condition}:  $K$-displacement convexity of the entropy  functional  ${\rm Ent}(~ \cdot ~| \mm)$  on $L^2$-Wasserstein space $\mathcal {W}_2(M)=(\mathcal{P}_2(M), W_2)$ (with respect to the Riemannian distance $\d_\g$). This means,  for any $\mu_0, \mu_1 \in \mathcal {P}_2 (M)$ with $\mu_0, \mu_1 \ll \mm$, there  is  a $L^2$-Wasserstein geodesic $(\mu_t)_{t\in [0,1]}$ such that 
\begin{equation}\label{eq1.5-intro}
\frac K2 t(1-t)W^2_2(\mu_0, \mu_1)+{\rm Ent}(\mu_t |\mm) \leq t{\rm Ent}(\mu_1 | \mm)+(1-t){\rm Ent}(\mu_0 | \mm)
\end{equation}
where ${\rm Ent}(\mu_t| \mm):=\int \ln \rho_t\,\d \mu_t$ if $\mu_t=\rho_t\,  \mm$.
\item [{\bf 3)}]  {Gradient estimate of heat semi-group}:
\begin{equation}\label{eq2-intro}
|\nabla \mathrm{H}_t(f) |^2 \leq e^{-2Kt}\mathrm{H}_t(|\nabla  f|^2)
\end{equation}
for any $f\in W^{1,2}(M,  \mm)$, where $(\mathrm H_t)_{t>0}$ is the semi-group generated by the diffusion operator   $\La$.
\end{itemize}

Let $\Omega \subset M$ be an open  set  with smooth boundary, and  $\d_\Omega$ be  (the completion of)  the intrinsic distance on $\overline \Omega$  induced by the Riemannian distance $\d_\g$.  Concerning  the metric measure space $(\overline \Omega, \d_\Omega, e^{-V}{\rm Vol}_{\rm g})$, one would ask the following questions.

 \begin{itemize}
\item [{\bf Q-1}]  What is Bakry-\'Emery's $\Gamma$-calculus on $(\overline \Omega, \d_\Omega, e^{-V}{\rm Vol}_{\rm g})$,  and what does  $\Gamma_2 \geq K \Gamma$ mean in this case?
\item [{\bf Q-2}]  What does $\cd$ condition \eqref{eq1.5-intro} imply? Can we say that $\Omega$ is geodesically convex?
\item [{\bf Q-3}]  Does  the gradient estimate \eqref{eq2-intro}  of  (Neumann) heat semi-group  imply geodesical convexity? 
\end{itemize}

\bigskip

Firstly, in Section \ref{section:ricci} we  study  the Bakry-\'Emery's $\Gamma$-calculus on smooth metric measure spaces with  smooth boundary.  Using the vocabularies  and results  developed in non-smooth theory (c.f. \cite{S-S} and \cite{G-N}),   we  define a measure-valued Ricci tensor $\bRic_\Omega$ by
\begin{equation}\label{eq:intro-1}
\bRic_\Omega(\cdot, \cdot)=\Ric_V(\cdot, \cdot)\,e^{-V} {\rm Vol}_{\rm g}+II (\cdot, \cdot)\,e^{-V} \mathcal{H}^{n-1}\restr{\partial \Omega}
\end{equation}
where $\Ric_V=\Ric+\H_V$ is the Bakry-\'Emery's modified Ricci tensor and $II$ is the second fundamental form. Combining with  the results in \cite{AGS-B, AGS-M} and \cite{G-N}, we  can see  that  the measure-valued Bochner inequality  $\bRic_\Omega \geq K$ is equivalent to the non-smooth  $\be$ condition and the  Lott-Sturm-Villani's $\cd$ condition.  More precisely,  recall that  a connected  Riemannian manifold with boundary is geodesically convex if and only if $II \geq 0$   (c.f. Theorem 1.2.1 \cite{WFY-A}),  we have  the following theorem.

\begin{theorem}[Measure-valued Ricci tensor, Theorem \ref{th:ricci} and Corollary \ref{coro:rigid}] \label{theorem1}
Let $(M,  {\rm g}, e^{-V}{\rm Vol}_{\rm g})$ be a $n$-dimensional weighted  Riemannian manifold and $\Omega \subset M$ be a  submanifold  with $(n-1)$-dimensional smooth  orientable  boundary. Then  Gigli's  measure-valued Ricci tensor (c.f. \cite{G-N}) on  $(\overline\Omega, \d_\Omega, e^{-V}{\rm Vol}_{\rm g})$ 
is given  by
\begin{equation}
\bRic_\Omega(\nabla g, \nabla g)=\Ric_V(\nabla g, \nabla g)\,e^{-V} {\rm Vol}_{\rm g}+II(\nabla g, \nabla g) \,e^{-V} \mathcal{H}^{n-1}\restr{\partial \Omega}
\end{equation}
for any  $g\in C_c^\infty$ with $ {\rm g} ({\rm N}, \nabla g)=0$, where ${\rm N}$ is the outward normal vector field on $\partial \Omega$.

Furthermore,   $(\overline \Omega, \d_\Omega, e^{-V} {\rm Vol}_{\rm g})$ is a $\cd$ space  if and only if $\Ric_V \geq K$,  $II \geq 0$ and path-connected. 
\end{theorem}

On the other side, from   \cite{AGS-C, AGS-M}
 we know that $(\overline \Omega, \d_\Omega, e^{-V} {\rm Vol}_{\rm g})$ is  $\cd$ if and only if the gradient estimate \eqref{eq2-intro} holds. It is proved by  F.-Y. Wang (c.f.  Chapter 3 in \cite{WFY-A}), that the  gradient estimate \eqref{eq2-intro} is equivalent to $\Ric \geq K$ and $II \geq 0$.   Thus we  completely answer the questions {\bf Q-1}, {\bf Q-2} and {\bf Q-3}.

\bigskip

In the discussions above, we have assumed  that $\partial \Omega$ is smooth  and there is no mass on the boundary. However,   none of these assumptions  is necessary or natural  in abstract metric measure setting. 
Precisely, for any   $\Omega$  and any  Borel measure $\mm$  with full support.      It is always meaningful to investigate  the  $K$-displacement convexity of   ${\rm Ent}(~ \cdot ~| \mm)$  on $L^2$-Wasserstein space $\mathcal {W}_2(\overline\Omega, \d_\Omega)$, 
which is   exactly  the approach used by Lott-Sturm-Villani to  define synthetic  lower Ricci curvature bound (i.e. $\cd$ conditions, c.f. \cite{Lott-Villani09, S-O1, S-O2}).

Let $(e_t)_{t\in [0,1]}$ denote the evaluation maps on a metric space $(X, \d)$: 
\[
  e_{t} : \geo( X, \d) \ni \gamma\mapsto \gamma_t \in  X.
\]
A measure $\Pi \in  \mathcal{P}(\geo( X, \d))$ is called an optimal dynamical plan if $(e_0,e_1)_{\sharp} \Pi$ is an optimal transportation plan; it easily follows that
$[0,1] \ni t \mapsto (e_t)_\sharp \Pi$ is a geodesic in $\mathcal {W}_2(X, \d)$. Conversely, it is known that any geodesic $(\mu_t)_{t \in [0,1]}$ in $\mathcal {W}_2(X, \d)$ can be lifted to an optimal dynamical plan $\Pi$ so that $(e_t)_\sharp \Pi  = \mu_t$ for all $t \in [0,1]$ (c.f. \cite[Theorem 2.10]{AG-U}). 

Now we  are ready to introduce the following   $\cdkn$ condition. 

\begin{definition} [Definition 1.3 \cite{S-O2}] Let $K \in \R$ and $N \in [1, \infty]$. We say that $(\overline \Omega, \d_\Omega,  \mm)$ is a $\cdkn$ space if  for {\bf any}  pair $\mu_0, \mu_1 \in \mathcal {P}_2 (\overline \Omega)$ with $\mu_0, \mu_1 \ll \mm$, there {\bf exists} a $L^2$-Wasserstein geodesic $(\mu_t)_{t\in [0,1]}$ connecting  $\mu_0, \mu_1$ such that 
\begin{equation}\label{eq3-intro}
{\rm S}_N(\mu_t | \mm) \leq -\int \Big [ \tau^{(t)}_{K,N} \big (\d_\Omega(\gamma_0, \gamma_1) \big )\rho^{-\frac 1N}_1(\gamma_1)+\tau^{(1-t)}_{K, N} \big (\d_\Omega(\gamma_0, \gamma_1) \big )\rho^{-\frac 1N}_0(\gamma_0) \Big ] \,\d \Pi(\gamma)
\end{equation}
for all $t\in [0,1]$ and some distortion coefficients $\tau^{(t)}_{K,N} $,  where $\Pi \in \mathcal{P}(\geo( \overline{\Omega}, \d_\Omega))$  is a lifting of $(\mu_t)$  satisfying  $(e_t)_\sharp \Pi=\mu_t$,  ${\rm S}_N(\mu_t| \mm)=-\int  \rho^{-\frac 1N}_t\,\d \mu_t$ and $\mu_t=\rho_t\,  \mm$.
\end{definition}
In addition,  the heat semi-group   on   $(\overline \Omega, \d_\Omega,  \mm)$ can be defined as the $L^2$-gradient flow of the energy form $\mathcal  E (f): W^{1,2}(\overline{\Omega}, \mm) \ni f \mapsto \int |\nabla f|^2\,\d\mm$.  So  gradient estimate \eqref{eq2-intro} is also well-posed for general $\Omega$ and $\mm$.

\bigskip

In \cite{AGS-M, AGS-B},  Ambrosio-Gigli-Savar\'e  prove that a  $\cd$  metric measure space whose Sobolev space $W^{1,2}$ is Hilbert (a notion named {\sl infinitesimally Hilbertianity} in the subsequent \cite{G-O}) is characterised by the validity of the gradient estimate \eqref{eq2-intro}.  

Therefore we have the following natural questions.

  \begin{itemize}
  \item [{\bf Q-4}]  If the boundary of $\Omega$  is not smooth or  $\mm\restr{\partial \Omega}\neq 0$,  such that $(\overline \Omega, \d_\Omega, \mm)$ is $\cd$,  can we say that $\Omega$ is geodesically convex as in Theorem \ref{theorem1}?
\item [{\bf Q-5}]  If  there exists a measure $\mm$ with full support such that $(\overline \Omega, \d_\Omega, \mm)$ is $\cd$ (or $\cdkn$). Can we assert $\mm \ll {{\vol}}_{\g}$?  Is  $(\overline \Omega, \d_\Omega, \mm)$ a  $\rcd$ space?
\end{itemize}

\bigskip

In Section \ref{section:rcd} we answer these  questions completely.  It should be noticed that Riemannian manifolds  with boundary are usually  non-smooth metric measure spaces, even if  the boundary is smooth. In the presence of an obstacle, the behavior of the  geodesics is quite involved.  For example,  the regularity of the  geodesics can not  be better than $C^1$ in general (c.f. \cite{AA-G} and \cite{ABB-R}).  In addition, in many  problems related to regularity,  (local)  convexity plays essential roles. Therefore it is difficult to solve this problem with classical analysis and (second order) PDE methods. 
However, using optimal transportation theory as an effective tool,   it will not be more difficult to study  non-smooth boundaries than smooth ones.

Firstly we show  the absolute continuity of the reference measure and the regularity of its density.  The following  theorem improves  the results proved by Cavalletti-Mondino \cite{CM-M} and Kell \cite{Kell-Transport} for essentially non-branching  ${\rm MCP}(K, N)$ spaces. To the knowledge of the author, this is the first measure rigidity result  without dimension bound.

\begin{theorem}[Measure rigidity:  absolute continuity and regularity, Proposition \ref{lemma:rcddensity}  \ref{prop:differentiable} \ref{prop:differentiableL1}]
Let $(M, {\rm g}, {\vol}_\g)$ be a complete $n$-dimensional Riemannian manifold without boundary,  $\d_\g$  be  the Riemannian distance  induced by ${\rm g}$ and $\mm$ be  a Borel measure with full support on $M$.    
Then we  have the following results.
\begin{itemize}
\item [1)]  Assume  $(M, \d_\g, \mm)$  satisfies the $\cd$ condition for some $K\in \R$. Then   there is a  locally Lipschitz semi-convex   functions $V$ on $M$  such that $\mm=e^{-V}{\rm Vol}_{\rm g}$.
\item [2)]  Assume  $(M, \d_\g, \mm)$  satisfies the $\mcpkn$ condition for some $K\in \R$ and  $N<\infty$.  Then $\mm=e^{-V}{\rm Vol}_{\rm g}$ for some locally bounded function $V$.
\end{itemize}

\end{theorem}

In the next  theorem,  we prove  that there is no non-trivial measure other than  volume measure on a $n$-dimensional Riemannian manifold,  such that the corresponding  metric measure space satisfies ${\rm CD}(k, n)$.    On a weighted Riemannian manifold with $C^2$ weight, it  is known that  the Bakry-\'Emery condition ${\rm BE}(k, n)$ holds if and only if the weight is a constant.  However in general case this is still an open problem.
Recently De Philippis-Gigli \cite{DPG-Non}  propose two  definitions of non-collapsed metric measure space. We say that a $\cdkn$ space $(X, \d, \mm)$ is  weakly non-collapsed if $\mm \ll \mathcal{H}^N$, and we call it   (strongly) non-collapsed if $\mm =c\mathcal{H}^N$ for some constant $c$. In   case  $X$ is a  Riemannian manifold,  the following Theorem tells us that non-collapsing  and weakly non-collapsing are equivalent. We remark that this  result  is also obtained   by  Kapovitch-Ketterer in \cite{KK19}  and   Kapovitch-Kell-Ketterer in  \cite{KKK19}, as a by-product  in studying RCD condition on Alexandrov spaces.

\begin{theorem}[Measure Rigidity: non-collapsed spaces, Theorem \ref{prop:noncollap}]\label{theorem2}
Let $(M, {\rm g})$ be a $n$-dimensional  Riemannian manifold.   Assume there exists a   measure $\mm$ with full support such that $(M, \d_\g, \mm)$ is  ${\rm CD}(k, n)$.
Then $\mm=c{\vol}_\g$ for some  constant $c>0$. 
\end{theorem}

In the last theorem, we study dimension-free $\cd$ condition on manifolds with possibly non-smooth boundary.    
We prove that the reference measure must be  supported on  a geodesically convex set, and we  answer  the question why there is no mass on the boundary. In particular,  we  fully understand the curvature-dimension condition on smooth metric measure space with boundary (c.f. Theorem \ref{th:ricci}). Thus this result enhance our understanding to  Cordero-Erausquin-McCann-Schmuckenschl\"ager's Riemannian interpolation inequality \cite{CMS01}.

In attacking this problem, we  assume {neither} infinitesimally Hilbertian nor non-branching property, which are often used in the study of related problems.  In particular,  we do not know a priori whether the $L^2$-Wasserstein geodesic is unique or not (i.e. the  Brenier-McCann's theorem).
In the proof,  we make full use of measure decomposition theory,  $L^1$-optimal transport theory and its connection with   $L^2$-optimal transport,  which are  developed by Klartag,  Cavalletti and Mondino (c.f. \cite{KlartagNeedle}, \cite{Cavalletti-D}, \cite{CM-ISO}).

\begin{theorem}[Measure rigidity: $\cd$ condition,     Theorem \ref{th:convex}]
Let $(M, {\rm g})$ be a complete  Riemannian manifold,  $\Omega \subset M$ be an open set with Lipschitz boundary.  Let  $\d_\Omega$ be the intrinsic distance induced by  Riemannian distance $\d_\g$ on $\overline{ \Omega}$, and $\mm$ be a reference measure with $\supp \mm=\overline \Omega$.  Assume that $\partial \Omega$ is $C^2$ out of a $\mathcal H^{n-1}$-negligible set, and  $(\overline{\Omega}, \d_\Omega, \mm)$  satisfies  $\cd$ condition for some $K\in \R$, then we have the following rigidity results.
\begin{itemize}
\item [1)]   $\overline \Omega$ is ${\rm g}$-geodesically convex,  that is,  any shortest path in  $(\overline { \Omega}, \d_\Omega)$ is a geodesic  segment  in $(M, {\rm g})$;
\item [2)]   $\mm \restr{\partial \Omega}=0$ and  $\mm=e^{-V}{\rm Vol}_{\rm g}$ for some  semi-convex, locally Lipschitz  function $V$;
\item [3)]  $(\overline{\Omega}, \d_\Omega, \mm)$  is a  $\rcd$ space.
\end{itemize}
In particular,   $(\overline{\Omega}, \d_\Omega, {\vol}_\g)$  is $\cd$  if and only if  $\overline \Omega$ is $\g$-geodesically convex and $\Ric \geq K$ on $\Omega$.
\end{theorem}

\bigskip

At last, we remark that most of the measure rigidity results obtained in this paper are still true on Alexandrov spaces with bounded curvature, which can be proved  in  similar ways.   Compared with previous results about  curvature-dimension conditions with finite dimension, we  use some new methods  to  deal with the  infinite dimensional  problem. We believe that these methods  have  potential  applications  on more general metric measure spaces.

\bigskip

\noindent \textbf{Acknowledgement}:  This research is part of a project which has received funding from the European Research Council (ERC) under the European Union's Horizon 2020 research and innovation programme (grant agreement No. 637851).  Part of Section \ref{section:ricci} had been finished when  the author  was supported by a postdoctoral  fellowship from Hausdorff Center for Mathematics (HCM)  in Bonn. The author wants to thank  E. Milman and M. Kell for helpful discussions and comments,  also  the anonymous referee for the  valuable report which improved readability.

\section{Smooth metric measure spaces with boundary}\label{section:ricci}
Let $\ms$ be an abstract metric measure space.
We say that $f\in L^2(X, \mm)$ is a Sobolev function in $W^{1,2}\ms$ if there exists a sequence of Lipschitz functions $(f_n)_n \subset L^2$,  such that $f_n \to f$ and $\lip{f_n} \to G$ in $L^2$ for some $G$, where $\lip{f_n}$ is the local Lipschitz constant of $f_n$ defined by
 \[
 \lip{f_n}(x):=\lmts{y}{x} \frac{|f_n(y)-f_n(x)|}{\d(y, x)} 
\]
(and $0$ if $x$ is an isolated point).   It is known that there exists a minimal function in $\mm$-a.e. sense, denoted it by $|\D f|$,  called minimal weak upper  gradient. 
If  $(X, \d)$ is a Riemannian manifold and $\mm$ is the volume measure, we know that $|\D f|=|\nabla f|=\lip{f}$  for any $f\in C^\infty$ (c.f. Theorem  6.1  \cite{C-D}).  Furthermore, let  $\Omega \subset X$ be an open set with  $\mm(\partial \Omega )=0$, by locality we have  $|\D f|_{\overline \Omega}=|\nabla f|$ $\mm$-a.e..

We equip $W^{1,2}\ms$ with the norm
\[
\|f\|_{W^{1,2}\ms}:= \sqrt{\|f\|^2_{L^2(X,\mm)}+\||\D f|\|^2_{L^2(X,\mm)}}.
\]
It is known that $W^{1,2}\ms$ is a Banach space, but not necessarily a Hilbert space. We say that $\ms$ is an {\bf infinitesimally Hilbertian space} if $W^{1,2}$ is a Hilbert space. Obviously, Riemannian manifolds equipped with volume measure are infinitesimally Hilbertian spaces.  In general, infinitesimal Hilbertianity is not  trivial even if the base space $X$  is a  manifold.

On an infinitesimally Hilbertian space, there a  pointwise bilinear map defined by
 \[
[W^{1,2}]^2 \ni (f, g) \mapsto \la \nabla f, \nabla  g \ra:= \frac14 \Big{(}|\D (f+g)|^2-|\D (f-g)|^2\Big{)}.
\]  It can be seen that $\la \cdot, \cdot \ra =\g(\cdot, \cdot)$ on a  Riemannian manifold $(M, \g)$.

Then we can define the measure-valued Laplacian via integration by part.

\begin{definition}
[Measure-valued Laplacian, \cite{G-O, G-N}]
The space ${\rm D}({\bf \Delta}) \subset  W ^{1,2}\ms$ is the space of $f \in  W ^{1,2} \ms$ such that there is a measure ${\bf \mu}$ satisfying
\[
\int h \,\d{\bf \mu}= -\int \la \nabla h, \nabla f \ra \, \d  \mm, ~~~\forall h ~ \text{Lipschitz with bounded support}.
\]
In this case the measure $\mu$ is unique and we shall denote it by ${\bf \Delta} f$. If ${\bf \Delta} f \ll m$, we denote its density by $\Delta f$.
\end{definition}

The following proposition links the curvature-dimension condition $\rcd$ and the non-smooth Bakry-\'Emery theory. We say that a space is  $\rcd$ if it is a  ${\rm CD}(K, \infty)$  space  defined by  Lott-Sturm-Villani in \cite{Lott-Villani09, S-O1, S-O2},  equipped with an infinitesimally Hilbertian Sobolev space. For more details, see \cite{AGS-M} (also  \cite{AGMR-R}).

We define  ${\rm TestF}\ms \subset W^{1,2}\ms$, the set of test functions  by
\[
{\rm TestF}\ms:= \Big\{f \in {\rm D} ({\bf \Delta}) \cap L^\infty: |\D f|\in L^\infty~~ {\rm and}~~~ \Delta f  \in W^{1,2}\ms\cap L^\infty(X, \mm) \Big \}.
\] It is known (c.f. \cite{S-S})  that ${\rm TestF}\ms$ is dense in $W^{1,2}\ms$ when $\ms$ is $\rcd$.

\bigskip

 Let $f,g \in {\rm TestF}\ms$. It is known from  \cite{S-S} that   measure ${\bf \Gamma}_2(f,g)$ is well-defined by
\[
{\bf \Gamma}_2(f,g)=\frac12 {\bf \Delta} \la \nabla f, \nabla g \ra -\frac12 \big{(}\la \nabla f, \nabla \Delta g \ra+\la \nabla g, \nabla \Delta f \ra\big{)}\, \mm,
\]
and we put ${\bf \Gamma}_2(f):={\bf \Gamma}_2(f,f)$. 
Then we have the following Bochner  inequality on metric measure spaces.

\begin{proposition} [Bakry-\'Emery condition, \cite{AGS-B, AGS-M}]\label{becondition}
Let $\ms$ be  a $\rcd$ space for some $K \in \R$. Then 
\[
{\bf \Gamma}_2(f) \geq  K |\D f|^2 \,\mm
\]
for any $f \in {\rm TestF}\ms$.
\end{proposition}

Let $f \in {\rm TestF}\ms$. The Hessian map  $\H_f: \big \{ \nabla g: g \in {\rm TestF}\ms \big \}^2 \mapsto L^0\ms$ is defined  by
\[
2\H_f(\nabla g,\nabla h)=\la \nabla g, \nabla \la \nabla f, \nabla h \ra \ra +\la \nabla h, \nabla \la \nabla f, \nabla g \ra \ra-\la \nabla f, \nabla \la \nabla g, \nabla h \ra \ra
\]
 for any $g, h \in {\rm TestF}\ms$. t can be proved (see Theorem 3.4  \cite{S-S}  and Theorem 3.3.8 \cite{G-N}) that  $\H_f(\cdot, \cdot)$ can be extended to a symmetric $L^\infty\ms$-bilinear map on the $L^2$-tangent module of $\ms$,  and is continuous with values in $L^0\ms$. On Riemannian manifolds,  $\H_f$ coincides with  ${\mathrm D}^2  f$,  which is the  Hessian (tensor) in classical sense.
 
Furthermore,  we have the following proposition,  due to Gigli \cite{G-N}.
 
 \begin{proposition}[Theorem 3.6.7 \cite{G-N}] \label{prop-ricci}
Let $M$ be a $\rcd$ space. Then  
\[
\bRic(\nabla f, \nabla f) \geq K|\D f|^2\,\mm
\]
for any $f\in {\rm TestF}\ms$,
where the measure-valued Ricci tensor $\bRic$ is defined by
\begin{eqnarray*}
\bRic(\nabla f, \nabla f):={\bf \Gamma}_2(f) - \|\H_f\|^2_{\rm HS} \, \mm.
\end{eqnarray*}
\end{proposition}

\bigskip

Now we introduce our first theorem.

\begin{theorem}[Measure-valued Ricci tensor]\label{th:ricci}
Let $(M, {\rm g}, e^{-V}{\rm Vol}_{\rm g})$ be a $n$-dimensional weighted  Riemannian manifold and $\Omega \subset M$ be a connected open set  with $(n-1)$-dimensional smooth  orientable  boundary. Then  the measure valued Ricci tensor on  $(\overline \Omega, \d_\Omega, e^{-V}{\rm Vol}_{\rm g})$ is given by
\begin{equation}
\bRic_\Omega(\nabla g, \nabla g)=\Ric_V(\nabla g, \nabla g)\,e^{-V}\d {\rm Vol}_{\rm g} \restr{\Omega}+II(\nabla g, \nabla g) \,e^{-V}\d \mathcal{H}^{n-1}\restr{\partial \Omega}
\end{equation}
for any  $g\in C_c^\infty$ with $ {\rm g} ({\rm N}, \nabla g)=0$, where ${\rm N}$ is the outward normal vector field on $\partial \Omega$, and $\Ric_V$ is the Bakry-\'Emery's generalized  Ricci tensor.
\end{theorem}
\begin{proof}
From integration by part formula (Green's formula) on a Riemannian manifold, we know
\[
\int {\rm g}( \nabla f, \nabla g )\,e^{-V}\d {\rm Vol}_{\rm g}=-\int f \Delta_V g\,e^{-V}\d {\rm Vol}_{\rm g}+\int_{\partial \Omega} f{\rm g}({\rm N}, \nabla g)\,e^{-V}\d \mathcal{H}^{n-1}\restr{\partial \Omega}
\]
for any $f, g \in C_c^\infty$, where $\Delta_V:=(\Delta-\nabla V) $ and  $\mathcal{H}^{n-1}\restr{\partial \Omega}$ is the $(n-1)$-dimensional Hausdorff measure on $\partial \Omega$.  

Hence  $g \in {\rm D}({\bf \Delta}_\Omega) $ and  the measure-valued Laplacian is given  by  the following formula
\[
{\bf \Delta}_\Omega g=\Delta_V g\,e^{-V}{\rm Vol}_{\rm g} \restr{\Omega}- {\rm g}( {\rm N}, \nabla g) \,e^{-V}\mathcal{H}^{n-1}\restr{\partial \Omega}.
\] 
 Therefore for any $g \in C_c^\infty$ with ${\rm g}(  {\rm N}, \nabla g)=0$ on $\partial \Omega$, we have $g \in {\rm TestF}(\Omega)$.

Now we can compute the measure-valued Ricci tensor. Let  $g \in C_c^\infty$ be a test function  with  ${\rm g}({\rm  N}, \nabla g)=0$ on $\partial \Omega$. We have
\begin{eqnarray*}
\bRic_\Omega(\nabla g, \nabla g)&=&\frac12 {\bf \Delta}_\Omega |\D g|_\Omega^2-\la \nabla g, \nabla \Delta_\Omega g\ra_\Omega \,e^{-V}{\rm Vol}_{\rm g}-\|\Hess_g\|^2_{\rm HS}\,e^{-V}{\rm Vol}_{\rm g}\\
&=&\frac12 { \Delta}_V |\nabla g|^2\,e^{-V}{\rm Vol}_{\rm g}-{\rm g}(\nabla g, \nabla \Delta_V g) \,e^{-V}{\rm Vol}_{\rm g}-\|\Hess_g\|^2_{\rm HS}\,e^{-V}{\rm Vol}_{\rm g}\\ 
&& -\frac12{\rm g}( {\rm N}, \nabla |\nabla g|^2) \,e^{-V} \mathcal{H}^{n-1}\restr{\partial \Omega}\\
&=&  \Ric(\nabla g, \nabla g)\,e^{-V} {\rm Vol}_{\rm g}+\H_V(\nabla g, \nabla g)\,e^{-V}{\rm Vol}_{\rm g}\\
&&~~~~~~~~~~~~~~~-\frac12{\rm g}( {\rm N}, \nabla |\nabla g|^2) \,e^{-V} \mathcal{H}^{n-1}\restr{\partial \Omega}\\
&=& \Ric_V(\nabla g, \nabla g)\,e^{-V} {\rm Vol}_{\rm g}-\frac12{\rm g}( {\rm N}, \nabla |\nabla g|^2) \,e^{-V} \mathcal{H}^{n-1}\restr{\partial \Omega},
\end{eqnarray*}
where we use Bochner's  formula at the third equality,  and $\Ric_V=\Ric+\H_V$ is  Bakry-\'Emery's generalized  Ricci tensor on weighted Riemannian manifold with weight $e^{-V}$.

By definition of second fundamental form, we have
\[
II(\nabla g, \nabla g)={\rm g}( \nabla_{\nabla g} {\rm N}, \nabla g)={\rm g}\big ( \nabla {\rm g}({\rm N}, \nabla g ), \nabla g \big )-\frac12 {\rm g}( {\rm N}, \nabla |\nabla g|^2) .
\]
Recall that  ${\rm g}( N, \nabla g)=0$ on $\partial \Omega$, we have  ${\rm g}( \nabla_{\nabla g} {\rm N}, \nabla g)=-\frac12 {\rm g}( {\rm N}, \nabla |\nabla g|^2)$.

Finally, we obtain
\begin{equation}
\bRic_\Omega(\nabla g, \nabla g)=\Ric_V(\nabla g, \nabla g)\, {\rm Vol}_{\rm g}+II (\nabla g, \nabla g)e^{-V} \, \mathcal{H}_{n-1}\restr{\partial \Omega}
\end{equation}
for any  $g\in C_c^\infty$ with ${\rm g}(N, \nabla g)=0$.
\end{proof}

\bigskip
At the last, we list some simple  applications of Theorem \ref{th:ricci} whose proof  can be found in  \cite{H-RRM}, see also Theorem 1.2.1 and Theorem 3.3.2 in \cite{WFY-A} for more details.

\begin{corollary}[Rigidity: convexity of the boundary ]\label{coro:rigid}
Let   $(\overline \Omega, \d_\Omega, e^{-V}{\rm Vol}_{\rm g})$ be a space as stated  in Theorem \ref{th:ricci}. Then it  is $\rcd$   if and only if $\partial \Omega$ is convex and $\Ric_V \geq K$ on $\Omega$. 
\end{corollary}

The next  result tells us  that the boundary does not influence the dimension bound of the  smooth metric measure space.
\begin{corollary}
A  $n$-dimensional  Riemannian manifold with  boundary is $\rcd$ if and only if it is ${\rm RCD}(K, n)$.
\end{corollary}

The last corollary characterizes metric measure spaces with upper Ricci  curvature bound, see also \cite{ES-Rigidity} for a  rigidity result concerning a different notion of upper Ricci bound.

\begin{corollary}[c.f. \cite{Oneill} p.104]
 If $\bRic_\Omega \ll {\vol}_\g$,   then $\partial \Omega$ is totally geodesic, i.e. $II\equiv 0$.
\end{corollary}

\section{Main results: measure rigidity theorems}\label{section:rcd}

At first, we  recall some fundamental properties of spaces satisfying Lott-Sturm-Villani's  synthetic lower Ricci curvature bound (including $\cd$, $\cdkn$ and $\rcd$ spaces).

\begin{itemize}
\item [1)] (Uniform density bound of intermediate measures on  $\cd$ spaces,   Lemma 3.1 \cite{Rajala12}.)
Let $\mu_0, \mu_1 \in \mathcal{P}(X)$  be  a pair of probability measures with bounded density   and so that $W_2(\mu_0, \mu_1) < \infty$. Suppose also that ${\rm diam}(\supp \mu_0 \cup  \supp \mu_1) < \infty$.
Then there exists a $L^2$-Wasserstein geodesic  $(\mu_t)$ connecting $\mu_0$ and $\mu_1$ such that  the densities $\{\frac{\d \mu_t}{\d \mm}\}_t$ are  uniformly bounded.

\item [2)] (Generalized  Brunn–Minkowski inequality on  $\cdkn$ spaces,   Proposition 2.1 \cite{S-O2}.)
Given $K, N\in \R$, with $N \geq 1$, we set for $(t, \theta) \in [0, 1] \times \R^+$,
\[
 \tau^{(t)}_{K, N}   \big(\theta):=\left\{\begin{array}{llll}
\infty &\text{if}~~ K\theta^2 \geq (N-1)\pi^2,\\
t^{\frac 1N} \Big (\frac{\sin(t\theta \sqrt{K/(N-1)})}{\sin(\theta \sqrt{K/(N-1)})}    \Big )^{1-\frac1N} &\text{if}~~ 0<K\theta^2<(N-1)\pi^2,\\
t  &\text{if}~~K\theta^2=0,\\
t^{\frac 1N} \Big (\frac{\sinh(t\theta \sqrt{-K/(N-1)})}{\sinh(\theta \sqrt{-K/(N-1)})}    \Big )^{1-\frac1N} &\text{if} ~~ K\theta^2<0.
\end{array}
\right.
\]

Then for any measurable sets  $A_0, A_1 \subset X$ with $\mm(A_0)+\mm(A_1)>0$,  $t\in [0,1]$ and
$N'\geq N$, we have
\begin{equation}\label{eq-BMI}
\mm \big (A_t \big)^{\frac 1{N'}}   \geq \tau^{(1-t)}_{K, N'}   \big(\Theta)\mm(A_0 \big)^\frac1{N'} +\tau^{(t)}_{K, N'}   \big(\Theta \big)\mm \big(A_1\big)^\frac1{N'},
\end{equation}
where $A_t$  denotes the set of points which divide geodesics starting in $A_0$ and ending in $A_1$ with ratio $\frac t{1-t}$ and where $\Theta$ denotes the minimal ($K\geq 0$) or maximal ($K<0$)  length of such geodesics.  In particular, when $A_0$ is a single point, we have the following $(K, N)$-measure contraction property  (or MCP$(K, N)$ condition  for short):
\begin{equation}\label{eq-MCP}
\mm \big (A_t \big)  \geq \Big [ \tau^{(t)}_{K, N}  \big(\Theta \big)\Big ]^{N}\mm \big(A_1\big).
\end{equation}

\item [3)] (Riemannian-Curvature-Dimension condtion, i.e. ${\rm RCD}$ condition,  see \cite{AGS-M}) We say that a space is $\rcd$ (or $\rcdkn$)  if it is an infinitesimally Hilbertian $\cd$ (or $\cdkn$ respectively) space. It is known that Riemannian manifolds with lower Ricci curvature bound, Ricci limit spaces,  and Alexandrov spaces with lower curvature bound are ${\rm RCD}$ spaces.

\cite[Proposition 3.1, Part a)]{Han-CDonRM}

 Let $\ms$ be a  $\rcd$ space.  For any  $L^2$-Wasserstein geodesic $(\mu_t)$  connecting $\mu_0, \mu_1 \ll \mm$,   there exists  $\Pi \in    \mathcal{P}(\geo(X,\d))$  such that  $(e_t)_\sharp \Pi=\mu_t$ (c.f.  Theorem 2.10 \cite{AG-U}).  By   essentially non-branching property of $\rcd$ spaces (c.f. \cite{RS-N}),   $\Pi$ is  concentrated on a set of non-branching geodesics,  where we say that a set $\Gamma \subset  \geo(X,\d)$  is non-branching  if  for any $\gamma^1, \gamma^2 \in \Gamma$, it holds:
\[
\exists t \in (0,1) ~~\text{s.t.}~ ~\forall  s\in [0, t]~ \gamma_t^1 =\gamma_t^2 \Rightarrow  \gamma_s^1 =\gamma_s^2, ~ \forall s\in [0,1].
\]
Furthermore,   there exits a unique  $L^2$-Wasserstein geodesic  connecting $\mu_0, \mu_1$ (c.f. \cite{RS-N}) which is induced by an optimal transport map. 
\end{itemize}

\subsection{Measure rigidity:  absolute continuity}

Let $(X, \d, \mm_1)$ and $(X, \d, \mm_2)$ be two metric measure spaces satisfying essentially non-branching MCP$(K, N)$  condition.  For  $N<\infty$,  Cavalletti and Mondino  (c.f.  Proposition 8.1 and Corollary 8.2 in \cite{CM-M})  prove the mutual absolute continuity of  the reference measures $\mm_1, \mm_2$ (see  \cite{Kell-Transport} for a different proof given by Kell).  In case  $X$  is a Riemannian manifold, we extend such result to  dimension-free $\cd$ condition  and  prove  a quantitative  density estimate under $\cdkn$ condition.

\begin{proposition}[Measure rigidity:  absolute continuity]\label{lemma:rcddensity}
Let $(M, {\rm g}, {\vol}_\g)$ be a complete $n$-dimensional Riemannian manifold,    $\d_\g$ be the Riemannian distance  induced by ${\rm g}$ and $\mm$ be a Borel measure with full support on $M$. 
Then we  have the following results.
\begin{itemize}
\item [a)]  Assume that  $(M, \d_\g, \mm)$  is  $\cd$  for some $K\in \R$. Then $\mm \ll {\vol}_\g$.
\item [b)]  Assume that  $(M, \d_\g, \mm)$  is  $\mcpkn$ for some $K\in \R$,  $N<\infty$. Then   we have
\[
 \frac {\mm \big (B_r(x) \big)}{{\rm Vol}_{\rm g} \big(B_r(x) \big)} \in L^\infty_\loc~~~~\text{uniformly~in}~r>0.
 \]
In particular,   $\mm= e^{-V} \,{\vol}_\g$ for some $e^{-V} \in L^\infty_\loc$.
\item [c)] Assume that  $(M, \d_\g, \mm)$  is $\mcpkn$ for some $K\in \R$, $N<\infty$.   Then   we have
 \[
  \frac {r^{N-n} {\vol}_\g \big(B_r(x) \big)}{\mm \big(B_r(x) \big)} \in L^\infty_\loc ~~~~\text{uniformly~in}~r>0.
 \]
In particular, when $N=n$, we know $\mm= e^{-V} \,{\vol}_\g$  for some
${V} \in L^\infty_\loc$.
\end{itemize}
\end{proposition}

\begin{proof}
\bigskip
{\bf Part  a):} 

Let $U \subset M$  be a  bounded and  geodesically convex open set. By  local finiteness of $\mm$ (c.f.  Theorem 4.24 \cite{S-O1}) we know $\mm(U)<\infty$.   By definition, $(\overline{U}, \d_\g, \mm\restr{U})$  is still a $\cd$ space.  

Let $\mm:= \mm_{\rm ac}+\mm_{\rm s}$ be the Lebesgue decomposition of $\mm$ with respect to the Riemannian volume measure $\vol_\g$. We will prove $\mm \ll \vol_\g$ in the following two steps. Firstly, we will show  $\mm_{\rm ac} \neq 0$. Then we will  prove $\mm_{\rm s}=0$.
\bigskip

{\bf Step 1}:  $\mm_{\rm ac} \neq 0$  on $U$.

Given  a parameter $r >0$,  we  define $\epsilon_r: U \mapsto \R^+$ by
\begin{equation}\label{eq1-lemma:rcddensity}
\epsilon_r(x):=\frac {\mm \big(B_{r}(x) \big)}{{\rm Vol}_{\rm g} \big(B_{r}(x) \big)}.
\end{equation}
Assume by contradiction that $\mm=\mm_{\rm s}$,  then for any  constant  $c>0$, there must be
\begin{equation}\label{eq: density0}
\mm \Big (\big \{ x\in U: \lmti{r}{0}\epsilon_r(x) \leq  c  \big \} \Big ) =   0, 
\end{equation}
otherwise we can prove  $\mm_{\rm ac}\neq 0$ by a standard measure-theoretic argument (using covering theorems). For any $(c, r)\in (0, +\infty) \times (0, 1)$, we  define the set $S(c, r) \subset U$ by
\[
S(c, r):= \Big \{x\in U:  \epsilon_s (x) \geq c ~~\text{for all}~~s\leq r \Big \}.
\]
From definition,  we have the following monotonicity 
\begin{equation}\label{eq1.1-lemma:rcddensity}
S(c, r_2) \subset S(c, r_1) ~~\text{for all}\text ~~ r_1 < r_2.
\end{equation}
It can also be seen that $$\Big \{ x\in U: \lmti{r}{0}\epsilon_r(x) >c \Big \} \subset \cup_{0<r<1} S(c, r) =\lmt{r}{0} S(c, r).$$   So by \eqref{eq: density0},  we obtain
\begin{equation}\label{eq1.2-lemma:rcddensity}
\mm \big(\lmt{r}{0} S(c, r)  \big)=\lmt{r}{0}\mm \big(S(c, r)  \big)=\mm(U).
\end{equation}


Combining \eqref{eq1.1-lemma:rcddensity},  \eqref{eq1.2-lemma:rcddensity}, and the assumption $\supp \mm =X$, we get 
\begin{equation}\label{eq1.52-lemma:rcddensity}
\lmt {r}{0}  \vol_\g \big (\overline{S(c, r)}  \big )  =  \vol_\g (U).
\end{equation}

Next we will deduce  contradiction from \eqref{eq1.52-lemma:rcddensity} and $\cd$ condition.

Fix a point $y_0 \in U$ with  $\lmt{r}{0}\epsilon_r(y_0) = \infty$. By Rauch's  (and Toponogov's) comparison theorem, there exists a small $R>0$  and constants $c_1, c_2>0$  such that 
\begin{equation}\label{eq1.6-lemma:rcddensity}
c_1 t \d_\g(y, z) <  \d_\g(\gamma^{xz}_t,  \gamma^{xy}_t) < c_2 t \d_\g(y, z)~~\forall t\in (0,1],~~\forall x, y, z \in B_{3R}(y_0), 
\end{equation}
and
 \begin{equation}\label{eq2.5-lemma:rcddensity}
 \d_\g(\gamma^{xz}_s,  \gamma^{xy}_t) >  c_1 s\land t \d_\g(y,z)  ~~~\forall s,t \in (0,1)
\end{equation}
for any  $x, y, z \in B_{3R}(y_0)$ with $\d_\g(x, y)=\d_\g(x, z)$, where $\gamma^{xz}$ denotes the geodesic from $x$ to $z$ and $\gamma^{xy}$ denotes the geodesic from $x$ to $y$.
Moreover,    the following comparison principle holds  for any geodesic $\gamma^1, \gamma^2$ with endpoints in $ B_{3R}(y_0) \subset U$:
\begin{equation}\label{eq2-lemma:rcddensity}
  \d_\g(\gamma^1_t, \gamma^2_t) <  c_2 \max  \Big \{ \d_\g(\gamma^1_0, \gamma^2_0), \d_\g(\gamma^1_1, \gamma^2_1) \Big \}~~~~\forall t \in (0,1).
\end{equation}

Let  $y \in  B_{3R}(y_0) \setminus \overline{B_{2R}(y_0)}$.   Consider the $L^2$-Wasserstein  geodesic $(\mu^{r, y}_t)_t$  from  $\mu^{r, y}_0:= \frac1{\mm\big(B_r(y_0)\big)} \mm \restr{B_r(y_0)}$ to $\mu^{r, y}_1:= \frac1{\mm\big(B_{r}(y)\big)} \mm \restr{B_{r}(y)}$.  By density bound  of intermediate measures on $\cd$ space (c.f. Lemma 3.1  \cite{Rajala12} ),  we get the following (uniform) estimate
\begin{equation}\label{eq3-lemma:rcddensity}
\mm\big(\supp \mu^{r, y}_t\big) \gtrsim \min \Big \{ \mm\big(B_r(y_0)\big), \mm\big(B_{r}(y)\big)  \Big \} ~~~~~\forall t\in [0, 1],
\end{equation}
where we adopt the notation $A \lesssim B$ if there is a constant $C>0$ such that $A <C B$.
Combining \eqref{eq3-lemma:rcddensity} and  the fact ${\vol}_\g(B_r) \gtrsim r^n $, we get
\begin{equation}\label{eq4-lemma:rcddensity}
\mm\big(\supp \mu^{r, y}_t\big) \gtrsim  r^n  \min \Big \{ \epsilon_r(y_0), \epsilon_r(y)\Big \}  ~~~~~\forall t\in [0, 1].
\end{equation}

Let  $T_t$  be the optimal transport map which induces  $(\mu^{r, y}_t)_t$.   By \eqref{eq2-lemma:rcddensity} we know 
\[
\d_\g(T_t(x),  \gamma^{y_0y}_t) \leq c_2  \max  \Big \{ \d_\g(x,  y_0), \d_\g(T_1(x), y) \Big \}\leq  c_2 r
\]
for any $x \in B_r(y_0)$, where $\gamma^{y_0y}$ is the geodesic from $y_0$ to $y$. Therefore
\begin{equation}\label{eq:epsilonnei}
\mu^{r, y}_t \big(B_{c_2 r}(\gamma^{y_0y}_t ) \big )=1 ~~~~~\forall t\in [0, 1].
\end{equation}
In particular,  $\cup_t \supp \mu^{r, y}_t  \subset \overline{(\gamma^{y_0y})_{c_2r}}$, where $  (\gamma^ {y_0y}_t)_{c_2r}$ is the $c_2 r$-neighbourhood of $\gamma^{y_0y}$ w.r.t.  Hausdorff distance. 

For any $c>0$, by \eqref{eq1.52-lemma:rcddensity},  there exists a small $r=r(c) \ll R$ such that 
\begin{equation}\label{eq1.75-lemma:rcddensity}
\frac{ \vol_\g \big (\overline{\{\epsilon_r \geq    c \} }  \cap   B_{3R}(y_0) \setminus {B_{2R}(y_0)}\big )}{ \vol_\g \big (   B_{3R}(y_0) \setminus {B_{2R}(y_0)}\big )} >\frac12.
\end{equation}

Consider the  projection map  ${\rm Prj}:  B_{3R}(y_0) \mapsto \partial B_{2R}(y_0)$  along the radius,  defined by   
$${\rm Prj}(x):=\exp_{y_0}\left (\frac{2R \exp^{-1}_{y_0}(x)}{|\exp^{-1}_{y_0}(x)|} \right).$$
 It is known that 
\begin{equation}\label{eq1.76-lemma:rcddensity}
 \vol_\g \Big (   B_{3R}(y_0) \setminus {B_{2R}(y_0)}\Big )=\int_{2R}^{3R} \mathcal H^{n-1} \big ( \partial B_t(y_0) \big )\,\d t.
\end{equation}
Hence by Fubini's theorem and \eqref{eq1.75-lemma:rcddensity},   there is $t_0 \in [2R, 3R]$ such that 
\begin{equation}\label{eq1.77-lemma:rcddensity}
 \mathcal H^{n-1} \Big (\overline{\{\epsilon_r \geq    c \} }\cap  \partial B_{t_0}(y_0) \Big ) >\frac 12  \mathcal H^{n-1} \Big ( \partial B_{t_0} (y_0) \Big ).
\end{equation}
In addition, by \eqref{eq1.6-lemma:rcddensity}, we have
\begin{equation}\label{eq1.78-lemma:rcddensity}
\mathcal H^{n-1} \restr{\partial B_{2R}(y_0) } \geq  c_1  \frac{2R}{t_0} \big ({\rm Prj} \big)_\sharp \Big ( \mathcal H^{n-1} \restr{\partial B_{t_0}(y_0) } \Big ) \geq \frac{2c_1}{3}  \big ({\rm Prj} \big)_\sharp \Big ( \mathcal H^{n-1} \restr{\partial B_{t_0}(y_0) } \Big )
\end{equation}
and
\begin{equation}\label{eq1.79-lemma:rcddensity}
\mathcal H^{n-1} \restr{\partial B_{2R}(y_0) } \leq  c_2  \frac{t_0}{2R} \big ({\rm Prj} \big)_\sharp \Big ( \mathcal H^{n-1} \restr{\partial B_{t_0}(y_0) } \Big ) \leq \frac{3 c_2}{2}  \big ({\rm Prj} \big)_\sharp \Big ( \mathcal H^{n-1} \restr{\partial B_{t_0}(y_0) } \Big )
\end{equation}

 Combining the estimates \eqref{eq1.77-lemma:rcddensity}, \eqref{eq1.78-lemma:rcddensity} and \eqref{eq1.79-lemma:rcddensity}, we get
 \begin{eqnarray*}
 && \mathcal H^{n-1} \restr{\partial B_{2R}(y_0) }  \Big ({\rm Prj}\big ( \overline{\{ \epsilon_r \geq    c \}} \cap   B_{3R}(y_0) \setminus B_{2R}(y_0) \big)  \Big )\\
 &{\geq}& \mathcal H^{n-1} \restr{\partial B_{2R}(y_0) }  \Big ({\rm Prj}\big (\overline{ \{ \epsilon_r \geq    c \}} \cap   \partial B_{t_0}(y_0)  \big)  \Big )\\
\text{by}~ \eqref{eq1.78-lemma:rcddensity}  &\geq& \frac{2c_1}{3} \mathcal H^{n-1}\restr{\partial B_{t_0}(y_0) }  \Big (\overline{\{ \epsilon_r \geq    c \}  }\cap \partial B_{t_0}(y_0)\Big)\\
 \text{by}~ \eqref{eq1.77-lemma:rcddensity}   &>&\frac{2c_1}{3}   \frac 12  \mathcal H^{n-1} \big ( \partial B_{t_0} (y_0) \big )\\
 \text{by}~ \eqref{eq1.79-lemma:rcddensity}  &\geq & \frac{c_1}{3}   \big ({\rm Prj}^{-1} \big)_\sharp \Big ( \frac{2}{3 c_2} \mathcal H^{n-1} \restr{\partial B_{2R}(y_0) } \Big )\big ( \partial B_{t_0} (y_0) \big ),
 \end{eqnarray*}
and finally
 \begin{equation}\label{eq1.8-lemma:rcddensity}
\mathcal H^{n-1} \big ({\rm Prj}\big (\overline{ \{\epsilon_r\geq    c \} }  \cap   B_{3R}(y_0) \setminus {B_{2R}(y_0)}\big )\big)\geq \frac 29  \frac {c_1}{c_2} \mathcal H^{n-1} ( \partial B_{2R}(y_0)).
\end{equation}

By \eqref{eq1.8-lemma:rcddensity} we can find  an integer   $N \gtrsim \frac 1 {r^{n-1}}$ which is  independent of $c$, and  $N$ points $\{y_1, y_2,..., y_N\} \subset \{ x: \epsilon_r(x) \geq    c \}   \cap   B_{3R}(y_0) \setminus {B_{2R}(y_0)}$,  such that $\{{\rm Prj}(y_1), {\rm Prj}(y_2),..., {\rm Prj}(y_N)\}$ are sparsely distributed on $\partial B_{2R}(y_0)$ satisfying the following estimate:
\[
\d_\g \Big( {\rm Prj}(y_i),  {\rm Prj}( y_j)\Big) >\frac{4c_2}{c_1} r  ~~~~~~~\forall 1\leq  i<j \leq N.
\]
So by  \eqref{eq2.5-lemma:rcddensity},  we have
 \begin{equation}\label{eq1.81-lemma:rcddensity}
\d_\g \Big (\gamma^{y_0 {\rm Prj}(y_i)}_s,  \gamma^{y_0 {\rm Prj}(y_j)} _t\Big)  >\frac{c_1}2 \frac{4c_2}{c_1}  r= {2 c_2}  r ~~\text{ for any}~~ s, t\in [\frac 12, 1].
\end{equation}
By \eqref{eq:epsilonnei}, we  also have
 \begin{equation}\label{eq1.82-lemma:rcddensity}
\mathop{\bigcup}_{t\in [\frac 12, \frac23]} \supp \mu^{r, y_i}_t  \subset \overline{ \Big (\gamma^ {y_0y_i}_t \restr{t\in [\frac 12, \frac23]}\Big)_{c_2r}} \subset \overline {\Big (\gamma^ {y_0 {\rm Prj}(y_i)}_t \restr{t\in [\frac 12, 1]}\Big)_{c_2r}}.  
\end{equation}
Combining \eqref{eq1.81-lemma:rcddensity} and \eqref{eq1.82-lemma:rcddensity}, we know
$\mathop{\bigcup}_{t\in [\frac 12, \frac23]} \supp \mu^{r, y_i}_t$, $i=1,...,N$ are  disjoint.

  Furthermore, consider  the following set
\[
N(y_i, r):=   \left \{{\bf t}=(t_1, t_2,...): {t_i\in \left[\frac 12, \frac23 \right]},  \supp \mu^{r, y_i}_{t_1}, \supp \mu^{r, y_i}_{t_2},... \subset B_{2R}(y_0)~\text{are}~\text{disjoint} \right \}.
\]
It can be seen that $\max_{{\bf t}\in N(y_i,r)} |{\bf t}| \gtrsim \frac 1 r$, $i=1,..., N$.

In conclusion, we can find  approximate $\frac 1{r^n}$ measures whose supports are disjoint in $B_{2R}(y_0)$.
Combining with \eqref{eq4-lemma:rcddensity} and local finiteness of $\mm$ (c.f.  Theorem 4.24 \cite{S-O1})  we obtain the following estimate
 \begin{equation}\label{eq5-lemma:rcddensity}
r^n  \min \Big \{ \epsilon_r(y_0), \epsilon_r(y_1), ..., \epsilon_r(y_N)\Big \}    \frac 1{r^{n}}  <C \mm\big(B_{2R}(y_0)\big)<\infty
\end{equation}
where $C$  is independent of $c$.  By the choice of  $\{y_1, ...,  y_N\}$, we know 
$$ \min \Big \{ \epsilon_r(y_1), ..., \epsilon_r(y_N)\Big \}  \geq c.
$$
Letting $c \to \infty$,  by \eqref{eq5-lemma:rcddensity} we get $\epsilon_r(y_0) < C \mm (B_{2R}(y_0))$,   which is  the contradiction.     Therefore $\mm_{\rm ac} \neq 0$.
\bigskip

{\bf Step 2}:  $\mm_{\rm s} =0$ on $U$.

We will prove  the assertion by contradiction.
Assume that $\mm\restr{U}$ is not absolutely continuous w.r.t. ${\vol}_\g$, then there exists a compact singular  set $\mathtt{N} \subset U$ such that $\mm(\mathtt{N})=\mm_{\rm s} (\mathtt{N})>0$ and ${\vol}_\g(\mathtt{N})=0$. 

Since $\mm_{\rm ac} \neq 0$,  there exists  a   bounded set $E_L$ with positive $\mm$-measure such that $\frac{\d \mm}{\d {\vol}_\g}< L$ on  $ E_L$.  Denote by $(\mu_t)$ the $L^2$-Wasserstein geodesic from $\mu_0:= \frac1{\mm(\mathtt{N})} \mm \restr{\mathtt{N}}$ to $\mu_1:= \frac1{\mm(E_L)} \mm \restr{E_L}$. By the choice of $E_L$,  we know $\mu_1 \ll {\vol}_\g$ with bounded density. By measure contraction property of  $(\overline{U}, \d_\g, {\vol}_\g)$,  we know $\mu_t \ll {\vol}_\g$   for any $t>0$. In particular $\mu_t(\mathtt{N})=0$, so there is a Borel set $A_t \subset \supp \mu_t$ such that $A_t \cap \mathtt N=\emptyset$ and $\mu_t(A_t)=1$.  However,  by Lemma 3.1 \cite{Rajala12} again, we   have  $\mu_t \leq C_1 \mm$  for some constant  $C_1>0$.  Next we will  show the contradiction using the  argument in  \cite{Kell-Transport} (c.f. Lemma 6.4 therein).  Given $\epsilon>0$, we know 
\[
A_t \subset \supp \mu_t  \subset  (\supp \mu_0  )_\epsilon= ( \mathtt{N} )_\epsilon
\]
for $t$ small enough. Then 
\begin{eqnarray*}
\mm ( \mathtt{N}) &=& \lmt{\epsilon} 0 \mm  \big (  (\mathtt{N} )_\epsilon \big ) \\
&\geq & \lmts{t}{0} \mm(\supp \mu_t)\\
&\geq&  \lmts{t}{0} \mm(A_t\setminus \mathtt{N})+\mm(\mathtt{N})\\
&=&  \lmts{t}{0} \mm(A_t)+\mm(\mathtt{N})\\
&\geq& \lmts{t}{0} \frac 1 C_1 \mu_t(A_t)+\mm(\mathtt{N})\\
&\geq& \frac 1{C_1} +\mm(\mathtt{N})
\end{eqnarray*}
which  is the contradiction.  

Finally, since the choice of  $U$ is arbitrary, we know $\mm  \ll {\vol}_\g$ on whole $M$.

\bigskip

{\bf Part b):}

Given $x\in M$. 
For any  $y \in B_{3R}(x) \setminus \overline{B_{2R}(x)}$,  let us consider the $L^2$-Wasserstein geodesic $(\mu^r_t)_t$  from  $\mu^r_0:= \frac1{\mm\big(B_r(x)\big)} \mm \restr{B_r(x)}$ to $\mu^r_1:= \delta_y$.  By  measure contraction property, we have  the following (uniform) estimate
\begin{equation}\label{eq6-lemma:rcddensity}
\mm\big(\supp \mu^r_t\big) \gtrsim  \mm\big(B_r(x)\big) ~~~~~\forall t\in [0,\frac 23].
\end{equation}
Combining the definition of $\epsilon_r$  \eqref{eq1-lemma:rcddensity} and \eqref{eq6-lemma:rcddensity}, we get 
\begin{equation}\label{eq7-lemma:rcddensity}
\mm\big(\supp \mu^r_t\big) \gtrsim  r^n \epsilon_r(x)  ~~~~~\forall t\in [0,\frac 23].
\end{equation}

As  previously shown  in Part 1),  there exist (approximate) $\frac 1{r^n}$ measures whose supports are disjoint.  Combining with \eqref{eq7-lemma:rcddensity} we get
 \begin{equation}\label{eq8-lemma:rcddensity}
\epsilon_r(x)=\big (  r^n \epsilon_r(x) \big ) \frac 1{r^n}   \lesssim \mm\big(B_{2R}(x)\big).
\end{equation}
Since $\mcpkn$ condition yields  measure doubling property,  $\mm(B_{2R}(x))$ is finite. Hence  $\epsilon_r \in L^\infty(U, \vol_\g)$ uniformly in $r$.

Letting $r \to 0$,   by Lebesgue differentiation theorem  there is $e^{-V}\in L^\infty_\loc(M, {\vol}_\g)$ such that   
\[
e^{-V(x)}= \lmt{r}{0} \frac {\mm\big(B_r(x)\big)}{{\rm Vol}_{\rm g}\big(B_r(x)\big)},~~~{\vol}_\g-\text{a.e.}~~x.
\]

\bigskip

{\bf Part c):}  

For   $x \in M$ and $0<r\ll1$, we define 
\begin{equation}\label{eq8.9-lemma:rcddensity}
 \delta_r(x):=\frac{r^{N-n} {\vol}_\g\big(B_r(x)\big)}{\mm\big(B_r(x)\big)}.
\end{equation}


Denote by  $(\nu^r_t)_t$ the  Wasserstein geodesic from  $\nu^r_0:= \frac1{{\vol}_\g\big(B_r(x)\big)} {\vol}_\g \restr{B_r(x)}$ to $\nu^r_1:= \delta_y$, with $0<r \ll R$.  By  measure contraction property (of  compact Riemannian manifolds),   we have
\[
{\vol}_\g\big(\supp \nu^r_t\big) \gtrsim \vol \big(B_r(x)\big)  ~~~~~\forall t\in [0,\frac 23].
\]

Combining with \eqref{eq8.9-lemma:rcddensity}, we obtain
\begin{equation}\label{eq9-lemma:rcddensity}
{\vol}_\g\big(\supp \nu^r_t\big) \gtrsim r^{n-N}\delta_r \mm\big(B_r(x)\big)  ~~~~~\forall t\in [0,\frac 23].
\end{equation}

By Bishop-Gromov inequality (c.f. Corollary 2.4 \cite{S-O2}), we know $\mm(B_r(x)) \gtrsim \big( \frac r {R} \big)^N\mm(B_{2R}(x))$. Therefore  \eqref{eq9-lemma:rcddensity} implies
\begin{equation}\label{eq10-lemma:rcddensity}
{\vol}_\g\big(\supp \nu^r_t\big) \gtrsim   \delta_r  r^n  ~~~~~\forall t\in [0,\frac 23].
\end{equation}
Similarly,  we can find approximate $\frac 1 {r^{n}}$ measures whose supports are disjoint inside $B_{2R}(x)$,  then we obtain
\begin{equation}\label{eq11-lemma:rcddensity}
\delta_r      \lesssim {\vol}_\g \big(B_{2R}(x)\big).
\end{equation}
 Then we obtain the following  $L^\infty_\loc$ estimate
 \[
  \frac {r^{N-n} {\vol}_\g\big(B_r(x)\big)}{\mm\big(B_r(x)\big)} \in L^\infty_\loc.
 \]


If $N=n$,  we have $\frac{ {\vol}_\g(B_r(x))}{\mm(B_r(x))} \in L^\infty_\loc$. Letting $r \to 0$,  by Lebesgue differentiation theorem we know 
\[
e^{V}= \lmt{r}{0} \frac {{\vol}_\g\big(B_r(x)\big)}{\mm\big(B_r(x)\big)} \in L^\infty_\loc.
\]
Combining with $e^{-V} \in L^\infty_\loc$, we get  $V \in L^\infty_\loc$ (see also  Proposition \ref{prop:differentiableL1}).

\end{proof}

\bigskip
The following result has been  proved in the Part 1) of the proof above (see also Lemma 6.4 \cite{Kell-Transport}). For convenience of  later applications,  we extract  it as a  separate lemma.

\begin{lemma}\label{lemma:sigular}
Let   $\mu_0, \mu_1$  be  two probability measures with compact support. Assume that  $ \mu_1  \ll {\vol}_\g$ and  $(\mu_t) \subset \mathcal W_2(M, \g)$ is  the  unique  $L^2$-Wasserstein geodesic   connecting  $\mu_0$ and  $\mu_1$.
If  there exists  a locally finite measure $\mm$  such that the density  functions $\frac{\d \mu_t}{\d \mm}, t\in [0,1]$ are  uniformly bounded. Then $\mu_0 \ll {\vol}_\g$.

\end{lemma}

\subsection{Measure rigidity:  regularity  of  density,  non-collapsing}

In the following two propositions, we will improve   the  regularity  of  density functions obtained in Proposition \ref{lemma:rcddensity}.

\begin{proposition}[]\label{prop:differentiable}
Let $V: M \mapsto \R \cup \{+\infty\}$  be an extended-valued function on  a complete Riemannian manifold $(M, {\rm g})$, such that  $(M, \d_\g,  e^{-V}{\rm Vol}_{\rm g})$ is a $\cd$ space. Then
$V$ has a semi-convex, locally Lipschitz  representative.  
\end{proposition}
\begin{proof}

Denote $\mm:=e^{-V}{\rm Vol}_{\rm g}$.
Let $\Omega \subset M$ be  a   convex, bounded open  set  such that points in $\Omega$ do not have cut-locus inside $\Omega$,  and let  $L>0$  be a constant such that  the sub-level set  $ \{V\leq L \} \cap \Omega$ has positive $\mm$-measure (and hence positive $\vol_\g$-measure). Let $\mathcal V$ be  a family of sets  of bounded eccentricity defined in Lemma \ref{lemma:bddecent}. Denote by $E_L$ the subset  of $ \{V\leq L \}$ which consists of density 1 points, i.e.
\begin{equation}\label{eq1-prop:diff}
E_L:=\left \{ x: \lmt{U \in \mathcal V}{x} \frac {\vol_\g \big (\{V \leq L\} \cap U \big)}{\vol_\g\big (U\big)}=1\right \}.
\end{equation}


Firstly we will  prove the following claim.

{\bf Step 1}: For any $x\in E_L$, there is $r_0>0$ such that  

\begin{equation}\label{eq:convexlevel}
E_L \cap B_{r_0}(x)= {{ {\rm Conv}\big(E_L \cap B_{r_0}(x) \big)}}^\circ
\end{equation}
where $ {{ {\rm Conv}\big(E_L \cap B_{r_0}(x)\big)}}^\circ$ is the interior points of  the convex-hull  of ${E_L \cap B_{r_0}(x)}$   in $\Omega$.     In particular, $E_L$ is open and connected, $V$ is locally bounded from above.

\bigskip

Given  $x \in E_L$, $v\in T_x M$  and $y:=\exp_x(\frac 12 v)$.   Let  $(\mu^{\delta}_t)_{t\in [0,1]}$  denotes the $L^2$-Wasserstein geodesic from $\mu^{ \delta}_0:= \frac 1{{\vol}_\g \big( E_L \cap A_0^{\delta}\big)} {\vol}_\g \restr{ E_L \cap A^{\delta}_0}$ to   $\mu^{{\delta}}_1:= \frac 1{{\vol}_\g\big( A^{ \delta}_1 \big)}{\vol}_\g \restr{A^{\delta}_1}$, where $A^{ \delta}_0=B_{\delta}(x), A^{\delta}_1=B_{\delta}\big(\exp_x(v)\big)$  are  geodesic balls with radius $ \delta >0$.
 
By direct computation (c.f. `(ii) $\Longrightarrow$ (i)'  in the proof of Theorem 1.1 \cite{SVR-T}), we have
\begin{equation}\label{taylor1}
{\rm Ent}(\mu^{ \delta}_0| \vol_\g)=-\ln \vol_\g \big (E_L\cap A^{\delta}_0\big)=-\ln c_n -n\ln \delta-\ln \frac{\vol_\g\big(E_L\cap A^{\delta}_0\big)}{\vol_\g \big( A^{\delta}_0\big)}+O(\delta^2),
\end{equation}
and
\begin{equation}\label{taylor2}
{\rm Ent}(\mu^{ \delta}_1| \vol_\g)=-\ln \vol_\g \big(A^{\delta}_1\big)=-\ln c_n -n\ln \delta+O(\delta^2)
\end{equation}
where  $c_n := \mathcal L^n (B_1)$ in $\R^n$.

Since $\Omega$ is compact,   we may assume  that  the sectional curvature is bounded from above by $\kappa>0$.   By  Rauch's comparison theorem, 
for $\delta \ll |v| \ll1$ we can find  a set 
$$A^{ \delta}_{\frac 12}= B_{\big(1+(\kappa+1/n)|v|^2/8\big)\delta}(y)$$ 
such that $\gamma_{\frac 12} \in A^{ \delta}_{\frac 12}$ for each minimizing geodesic $\gamma: [0,1]\mapsto M$ with $\gamma_0 \in A^\delta_0, \gamma_1\in A^\delta_1$. In particular,  $\supp \mu^{\delta}_{\frac 12} \subset A^{ \delta}_{\frac 12}$.  Moreover,  by the asymptotic expansion formula in \cite[Page 929, Proof of Theorem 1.1]{SVR-T}
\begin{eqnarray*}
&&{\rm Ent}(\mu^{\delta}_{\frac 12}| \vol_\g)\\
 &\geq& -\ln \vol_\g(A^{ \delta}_{\frac 12})~~~~~~(\text{by Jensen's inequality} )\\
&\geq&-\ln c_n -n\ln \delta-\frac{(n\kappa+1)|v|^2}8+O(\delta^2)+O(|v|^4)~~~~(\text{c.f.  page 929 in \cite{SVR-T})}.
\end{eqnarray*}
Similarly,  by the expansion formulas \eqref{taylor1} and \eqref{taylor2},  for $k:={n\kappa}+2$, we have
\begin{eqnarray*}
&& \frac 12 {\rm Ent}(\mu^{\delta}_1| \vol_\g)+\frac 12 {\rm Ent}(\mu^{ \delta}_0| \vol_\g)-\frac{k}8 W^2(\mu^{ \delta}_{0}, \mu^{\delta}_{1})\\
 &=&-\ln c_n -n\ln \delta-\frac 12  \ln \frac{\vol_\g \big(E_L\cap B_\delta(x)\big)}{\vol_\g\big( B_\delta(x)\big)}- \frac{k|v|^2}8 +O(\delta).
\end{eqnarray*}

Thus for $\delta$ and $|v|$ small enough, we have
\begin{equation}\label{taylor3}
{\rm Ent}(\mu^{\delta}_{\frac 12}| \vol_\g)\geq \frac 12 {\rm Ent}(\mu^{\delta}_1| \vol_\g)+\frac 12 {\rm Ent}(\mu^{ \delta}_0| \vol_\g)-\frac{k}8 W^2(\mu^{ \delta}_{0}, \mu^{\delta}_{1}).
\end{equation}

By the  definition of  $\cd$,   ${\rm Ent}(~\cdot~ | \mm)$ is  $K$-convex along $(\mu^{ \delta}_t)$. Combining with  the inequality \eqref{taylor3} and the  following formula
\begin{equation*}
{\rm Ent}(~\cdot~ | \mm)={\rm Ent}(~\cdot~ | {\rm Vol}_\g)+\int V\,\d(\cdot),
\end{equation*}
we obtain 
\begin{eqnarray}\label{eq1-lemma:diff}
&& \int V\,\d \mu^{\delta}_{\frac 12 }-\frac{k-K}8W^2(\mu^{\delta}_{0}, \mu^{ \delta}_{1})\\ &\leq&  \frac 12\int V\, \d \mu^{ \delta}_1+\frac 12\int V\,\d \mu^{\delta}_0.
\end{eqnarray}
By replacing $V$ with $V+H$ for some locally $|K-k|$-convex function $H$ (and simultaneously replacing $\{V\leq L\}$ by $\{V+H\leq L\}$), we may assume $k=K$ without loss of generality.

By \eqref{eq1-lemma:diff} and $E_L \subset \{ V\leq L\}$,  we have
\begin{equation}\label{eq1.1-lemma:diff}
\int V\,\d \mu^{\delta}_{\frac 12}\leq  \frac 12\int V\, \d \mu^{ \delta}_1+\frac 12L.
\end{equation}
If  $y=\exp_x (\frac 12 v)$ is a density 1 point of $\{V>L\}$, by\eqref{eq1.1-lemma:diff} we get
\[
L< \lmti{\delta}{0}\int V\,\d \mu^{\delta}_{\frac 12}\leq  \lmti{\delta}{0} \frac 12\int V\, \d \mu^{ \delta}_1+\frac 12L,
\]
thus  
\begin{equation}\label{eq1.2-lemma:diff}
\gamma_1=\exp_x (v) \notin E_L.
\end{equation}

 Let $r_0>0$  be  small so that the geodesic ball $B_{r_0}(x)$ is geodesically convex.   Assume by contradiction that $\vol\Big ({\rm Conv}\big(E_L \cap B_{r_0}(x)\big)  \setminus E_L\Big )\neq 0$.  By Lebesgue density theorem,  the  density 1 points of  $\{V>L\}$ in $ B_{r_0}(x)$ are not negligible. 
 By  Fubini's theorem,   there is  $\gamma\subset B_{r_0}(x)$ such that $\gamma_0, \gamma_1 \in E_L$ and $\gamma_\frac 12$ is a  density 1 point of $\{V>L\}$, which contradicts to \eqref{eq1.2-lemma:diff}.

  Notice that  by Fubini's theorem we have $\vol_\g\Big (\partial {\rm Conv}\big(E_L \cap B_{r_0}(x)\big)\Big)=0$.  Hence
\[
E_L\cap B_{r_0}(x)=\mathop {{\rm Conv}\big(E_L  \cap B_{r_0}(x)\big )}^\circ 
\]
which is the thesis.

\bigskip

We  define an extended real-valued function $\bar V: M \mapsto \R \cup \{\pm \infty\}$ by
 \[
\bar V(x):= \lmti {\mathcal {V} \ni U}{x}  \frac1{{\vol}_\g(U)} \int_{U}V\,\d {\vol}_\g
\]
where $\mathcal {V} $ is the family of sets defined in Lemma \ref{lemma:bddecent}.

 Denote   $V^+:=V \vee 0$ and $V^-:=V \land  0$.  By the inequality $t \leq e^t$ on $[0, +\infty)$, we know $|V^-| \leq e^{-V} $ and $V^- \in L^1_\loc({\vol}_\g)$.  Combining with the fact that  $V^+ \leq L$ on $E_L$,  we get  $V\in L^1(E_L, {\vol}_\g)$.  By Lebesgue differentiation theorem,  we know $\bar V=V$ $\mm$-a.e.  on $E_L$ and  there is  $M^* \subset E_L$ with full measure such that

 $$
 \lmt{\mathcal {V} \ni U}{x} \frac1{{\vol}_\g(U)} \int_{U } V\,\d {\vol}_\g = \bar V(x) \in \R~~~~\forall x \in M^*.
 $$

{\bf Step 2}: $\bar V$ is  geodesically convex on $E_L$.  Then from  \cite{GW-convex} (see also Corollary 3.10 \cite{Constantin94}) we know  $\bar V$ is locally Lipschitz on $M$.

Let $\gamma \subset E_L$ be a geodesic with end points in $M^*$ and with small length,  and let $(\mu^\delta_t)$ be the geodesic defined similarly as in {\bf Step 1}. Similar to   \eqref{eq1-lemma:diff},  we have 
\begin{equation}\label{eq1.5-lemma:diff}
\int V\,\d \mu^{ \delta}_\frac 12\leq  \frac 12\int V\,\d \mu^{ \delta}_1+\frac 12\int V\,\d \mu^{\delta}_0.
\end{equation}

Letting $\delta \to 0$ in \eqref{eq1.5-lemma:diff},  we obtain
\begin{equation}\label{eq4-lemma:diff}
\bar V(\gamma_\frac 12)\leq   \frac 12  \bar V(\gamma_0)+\frac 12\bar V(\gamma_1)\leq L.
\end{equation}

  Given  $x\in E_L$, by Fubini's theorem we know there exists  $S_x \subset \{v\in S^1(T_x M)\}$ with full $\mathcal{H}^{n-1}$-measure and a  positive constant $\tau(x) \in  (0, 1]$,  such that  $\exp_x{(tv)}   \in  M^*$ for all $v \in S_x$ and  $\mathcal L^1$-a.e. $t\in (-\tau(x), \tau(x))$. 
Define a set of geodesic segments  $\Gamma_x$ by
 \[
 \Gamma_x:=\left \{\big(  \exp_x{(tv)}\big)_{t\in (-\tau(x),\tau(x))}:  v \in S_x \right \}.
\]
Then  for any $\gamma \in  \Gamma_x$,  \eqref{eq4-lemma:diff}  yields  that $\bar V$ is convex on   $\gamma  \cap M^*$. In particular, $\bar V$ is Lipschitz  on   $\gamma  \cap M^*$.

To prove the geodesical  convexity of $\bar V$ on whole $E_L$, we  just need to  show  the continuity of $\bar V$.   Then by an approximation argument,  we can see that  $\bar V$ satisfies \eqref{eq4-lemma:diff} on all geodesics.    With this aim, we will prove the following claims.

{\bf Claim 1}:  For any  $x\in E_L$,  $\Big \{\Lip \big(\bar V \restr{(\gamma_t)_{t\in [-\frac {\tau(x)}{ 2}, \frac {\tau(x)}{ 2}]} \cap M^*}\big)\Big \}_{ \gamma\in \Gamma_x}$ is bounded.

   By   \eqref{eq4-lemma:diff}  and the discussion thereafter, we know 
$$\bar V(\gamma_{0^+}):= \lmt{ y \in \gamma \cap M^*}{x} \bar V(y)\in [\bar V(x), L]
$$  for any $\gamma \in \Gamma_x$.
Thus  by convexity we know $\bar V$ is $\frac {2(L-\bar V(x)+1)}\tau(x)$-Lipschitz  on  $\Big \{ (\gamma_t)_{t\in [-\frac {\tau(x)}{ 2}, \frac {\tau(x)}{ 2}]} \cap M^*: \gamma\in \Gamma_x \Big\}$, which is the thesis.

\bigskip

We define a (possibly multi-valued)  function $\bar V':  E_L \mapsto [\bar V(x), L]$ by
\[
 \bar V'(x):=\Big\{
\bar V(\gamma_{0^+}): \gamma \in \Gamma_x\Big \}.
\]

  For any   $x \in M^*$,  by \eqref{eq4-lemma:diff}  we know the value of  $\bar V(\gamma_{0^+})$  is independent of the choice of the geodesic $\gamma \in \Gamma_x$ and  $\bar V(\gamma_{0^+})=\bar V(x)$, so $\bar V'= V$ almost everywhere.  Furthermore,  if we can show that  $\bar V'$ is continuous,  by definition we know $\bar V=\bar V'$ on  $E_L$. Therefore  it suffices it  to prove the following assertion.

{\bf Claim 2}: $\bar V'$ is single-valued and continuous on $E_L$.

Given  $x\in E_L$.    Assume by contradiction that $\bar V'$ is not single-valued, 
\begin{equation}\label{eq7-lemma:diff}
-\infty<\bar V(x)\leq \bar V(\gamma^1_{0^+})<\bar V(\gamma^2_{0^+})\leq L
\end{equation}
 for some different  $\gamma^1, \gamma^2 \in \Gamma_x$. By Fubini's theorem,  we can find  sequences $(x_n) \subset \gamma^1 \cap M^*$,  $(y_n) \subset \gamma^2\cap M^*$   such that 
$x_n, y_n \to x$, and  $y_n \in \gamma^{x_n} \in \Gamma_{x_n}$. 
From \eqref{eq4-lemma:diff},  we can see that  $\bar V'(x_n) \to \bar V(\gamma^1_{0^+})$, $\bar V'(y_n) \to \bar V(\gamma^2_{0^+})$,  and $\bar V' $ is Lipschitz on $\gamma^{x_n} \cap M^*$.
From   \eqref{eq7-lemma:diff} we also  know $\Lip \bar V' \restr{\gamma_{x_n}} \to +\infty$, which  contradicts to the fact that  $\bar V'\leq L$ on $E_L$.

Finally, let $(z_n) \subset E_L$ be an arbitrary sequence with $z_n \to x$.  By definition, we can find $z_n' \in M^*$ such that  $\d_\g(z_n',  z_n) <\frac 1n$ and $|\bar V(z_n')-\bar V'(z_n)| <\frac 1n$. By uniqueness  of $\bar V'(x)$ and  {\bf Claim 1},  we know $\bar V(z_n') \to \bar V'(x)$. So $\bar V'$ is continuous at $x$.

\end{proof}

\begin{lemma}\label{lemma:bddecent}
We define  a family of  measurable sets  $\mathcal V$ in the following way.  We say that $U \in \mathcal V$ if there are $x, y \in M$, $\delta>0$, and a $L^2$-Wasserstein geodesic $(\mu^{\delta}_t)_{t\in [0,1]}$ from $\mu^{ \delta}_0:= \frac 1{{\vol}_\g \big( B_\delta(x) \big)} {\vol}_\g \restr{ B_\delta(x)}$ to   $\mu^{{\delta}}_1:= \frac 1{{\vol}_\g\big( B_\delta(y)\big)}{\vol}_\g \restr{B_\delta(y)}$, such that $U=\supp \mu^\delta_\frac 12$.

Then $\mathcal V$  is a  fine covering with  bounded eccentricity. This means that every point $x \in M$ is covered by  sets in 
${\mathcal {V}}$ with arbitrarily small diameter, and  there exists a constant $c >0$ such that each set $U \in \mathcal {V}$ is contained in a ball  $B_r$ and $\vol (U)\geq c \vol (B_r)$. 
\end{lemma}
\begin{proof}

Let  $(T_t)$ be a family of maps which induce $(\mu^\delta_t)$, i.e.  $\mu^\delta_t=(T_{t})_\sharp \mu^\delta_0$.   Denote the geodesic from $x$ to $y$ by $\gamma$.  On one hand, by Rauch's comparison  theorem,   there  exists a  constant  $C_1>0$ such that
\[
  \d_\g \big(\gamma_t,  T_{t}(z)\big) \leq C_1 \Big (\d_\g(\gamma_0,  z)  +\d_\g(\gamma_1,  T_1(z))\Big)\leq  2C_1 \delta
\]
for any $z\in \supp \mu^\delta_0=B_\delta(x)$. So we have
$$
 \supp \mu^\delta_t \subset B_{2C_1 \delta}(\gamma_t).
$$
On the other hand,   Riemannian manifolds are locally CD spaces, hence
$$
{{\vol}_\g \big(  \supp \mu^\delta_t \big)} \gtrsim \Big \{{{\vol}_\g \big( B_\delta(x) \big)}, {{\vol}_\g \big( B_\delta(y) \big)} \Big \}.
$$
Therefore by Bishop-Gromov inequality, the sets in $\mathcal V$ have bounded eccentricity.
Furthermore, let $x=y\in M$,  it can be seen that  $U=B_\delta(x)$, so $\mathcal V$  is a fine covering of $M$.
\end{proof}

\bigskip

\begin{proposition}[]\label{prop:differentiableL1}
Let $V: M \mapsto \R \cup \{+\infty\}$  be an extended-valued function on  a compact  Riemannian manifold $(M, {\rm g})$ with boundary, and $\mm=e^{-V}{\rm Vol}_{\rm g}$.    If   $(M, \d_\g,  \mm)$ satisfies  $\mcpkn$  for some $K\in \R$ and $N<\infty$.  Then 
$V$ is locally bounded in the interior of $M$.  In particular,    $(M, \d_\g,  \mm)$  is infinitesimally Hilbertian.  
\end{proposition}
\begin{proof}

We define the following family of functions with parameter $r\in (0, 1)$,  as in the proof of  Proposition \ref{lemma:rcddensity}, 
\begin{equation*}
\epsilon_r(x):=\frac {\mm(B_{r}(x))}{{\rm Vol}_{\rm g}(B_{r}(x))}.
\end{equation*}

Given  $x\in M$ and $R>0$  with $B_R(x) \subset M$ and  $\lmt{r}{0} \epsilon_r(x)=e^{-V(x)}$. We  define a family $\mathcal V$ of Borel sets in the following way.  We say that $U \in \mathcal V$ if there exist  $0<r \ll \frac R2$, $x_0  \in B_{2R}(x) \setminus {B_R(x)}$,  
and a  $L^2$- Wasserstein geodesic $(\mu_t)_{t\in [0,1]}$ with $\mu_0:=\delta_{x_0}$ and $\mu_s:=\frac 1{\mm\big(B_r(x)\big)} \mm \restr{B_r(x)}$ for some $s\in [\frac13, 1]$,  such that $U=\supp \mu_1$.
For the similar reason as Lemma \ref{lemma:bddecent},   we know  the sets in $\mathcal V$ is a covering of $H_R(x):=B_{2R}(x) \setminus {B_R(x)}$ with  bounded eccentricity.

By Lebesgue differentiation theorem,   there exists $H_R^*(x) \subset H_R(x)$ with full measure such that 
 $$
  \lmt{\mathcal {V} \ni U}{y} \frac{\mm(U)}{{\vol}_\g(U)}=\lmt{\mathcal {V} \ni U}{y} \frac1{{\vol}_\g(U)} \int_{U } e^{-V}\,\d {\vol}_\g = e^ {-V(y) }>0~~~~\forall y \in H_R^*(x).
 $$

For any $y\in H_R^*(x)$ and $0<\delta \ll 1$. There is  a    $L^2$- Wasserstein geodesic $(\mu_t)_{t\in [0,1]}$  with  $\mu_0:=\delta_{x_0}$ and  $\mu_s:=\frac 1{\mm(B_r(x))} \mm \restr{B_r(x)}$ for some $x_0  \in B_{2R}(x) \setminus {B_R(x)}$, $s\in [\frac13, 1]$ and $r=r(\delta)$,  such that  $U=\supp \mu_1$ and
\[
1-\delta <\left |\frac{\frac{\mm(U)}{{\vol}_\g(U)}}{e^ {-V(y)} } \right |<1+\delta.
\]

By measure contraction property, there is a universal constant $C>0$ such that
$$
\epsilon_r(x) {\vol}_\g\big (B_r(x)\big )= \mm\big (B_r(x) \big ) > C \mm(U) > C (1-\delta) {\vol}_\g(U) e^{-V(y)}.
$$
Dividing $r^n$ on both sides and letting $r \to 0$, we get $e^{-V(x)} \gtrsim e^{-V(y)}$.  Recall that $H_R^*(x)$ has full measure in  $H_R(x)$,   we have the following weak mean-value property
\[
e^{-V(x)} \gtrsim \mm\big (B_R(x) \big )>0.
\]

Combining with Proposition \ref{lemma:rcddensity}, we know $V \in L^\infty_\loc$.

\end{proof}

\bigskip

Next we will prove that there is no non-trivial measure other than the volume measure such that a $n$-dimensional  Riemannian manifold satisfies  ${\rm CD}(K, n)$ condition. We remark that this result can also be obtained   by combining  Kapovitch-Ketterer's recent  result \cite[Corollary 1.2]{KK19}, and  Cavalletti-Mondino's  result   \cite[Corollary 8.3]{CM-M} about measure rigidity on Alexandrov spaces (see also   \cite[Theorem 1.4]{KKK19} for a more recent proof).

\begin{theorem}[Measure Rigidity: non-collapsed spaces]\label{prop:noncollap}
Let $(M, {\rm g})$ be a $n$-dimensional  Riemannian manifold.   Assume there exists a  measure $\mm^*$ with full support such that $(M, \d_\g, \mm^*)$ is   ${\rm CD}(K, n)$ for some $K\in \R$.
Then there  exists  a  constant $c>0$ such that $\mm^*=c {\vol}_\g$.
\end{theorem}
\begin{proof}
For any $x\in M$,  we can find a  small convex neighborhood  $U$ of it,  such that $(\overline {U}, \d_\g, \mm^*)$ is still  ${\rm CD}(K, n)$ and  $(\overline{U}, \d_\g, \vol_g)$ is a CD$(k, n)$ for some $k\in \R$.  So without loss of generality,  we may assume that any point in $M$ has no cut-locus  and  $(M, \d_\g, {\rm Vol}_{\rm g})$ is  ${\rm CD}(k, n)$ for some $k$.
By Proposition \ref{lemma:rcddensity} and Proposition \ref{prop:differentiable} we know there exists a  positive continuous function  $\varphi$ such that  $\mm^*=\varphi^n {\vol}_\g$. Hence we just need to prove that $\varphi$ is a constant.

Given two points  $x, y \in M$. Let  $\gamma$ be a geodesic from $x$ to $y$.     Let  $\mm=\mm^*=\varphi^n {\rm Vol}_{\rm g}$ and  $\mm= {\rm Vol}_{\rm g}$ respectively.  By  Brunn–Minkowski inequality on ${\rm CD}(K, n)$ spaces (c.f. \cite[Proposition 2.1]{S-O2} ) and Rauch's comparison theorem,   there is $C>0$ such that
\begin{equation}\label{eq1-prop:noncollapsed}
\mm \Big (B_{\frac \sigma2 \big(1+C\d^2_\g(x, y)\big)}(\gamma_\frac12) \Big)^{\frac 1n}   \geq \tau^{(\frac 12)}_{K, n}  (\Theta)\mm \big (B_\sigma(x) \big)^\frac1n
\end{equation}
where  $0<\sigma\ll 1$ and  $ \big | \Theta - \d_\g(x, y) \big |\leq \sigma$.
 
We define  $\mathcal{J}_\mm(x)$ by
\[
\mathcal{J}_\mm(x):=\lmt{r}{0} \left (\frac{\mm\big(B_r(x)\big)}{r^n} \right )^{\frac1n}.
\]
Dividing $\frac\sigma2 \big(1 +C\d^2_\g(x, y)\big)$ on both sides of \eqref{eq1-prop:noncollapsed} and letting $\sigma \to 0$, we obtain
\[
 \mathcal{J}_\mm(\gamma_\frac12 ) \geq \frac 2{1 +C\d^2_\g(x, y)} \tau^{(\frac12)}_{K, n}  \big(\d_\g(x, y)\big) \mathcal{J}_\mm(x).
\]

When $\d_\g(x, y)$ is small, by  Taylor expansion of $\tau^{(\frac 12)}_{K, n}  (\theta)$ we obtain
\[
 \mathcal{J}_\mm(\gamma_\frac 12) \geq \frac {1+O\big(\d^2_\g(x, y)\big)} {1+C\d^2_\g(x, y)} \mathcal{J}_\mm(x).
\]

For any $N>0$, we divide $\gamma$ equally  into $N$ parts. Repeating the argument above on each interval with length $\frac 1N \d_\g(x, y)$ we  get
\[
 \mathcal{J}_\mm(\gamma_\frac {i+1}{N}) \geq \Big ( {1 +o(\frac 1{N})}\Big ) \mathcal{J}_\mm (\gamma_{\frac{i}{N}})~~~~~i=0,..., N-1.
\]
Therefore
\[
 \mathcal{J}_\mm(y) \geq  \Big ( {1 +o(\frac 1{N})}\Big )^N \mathcal{J}_\mm(x).
\]
Letting $N \to \infty$, we obtain $ \mathcal{J}_\mm(y)  \geq  \mathcal{J}_\mm(x)$.   
By symmetry,  we can also prove $ \mathcal{J}_\mm(y)  \leq  \mathcal{J}_\mm(x)$, hence $ \mathcal{J}_\mm(y) = \mathcal{J}_\mm(x)$.  So  $\mathcal{J}_\mm$ is a constant for 
both $\mm ={\vol}_\g$ and $\mm=\mm^*$.  
By Proposition \ref{prop:differentiableL1} we know $\varphi$ is continuous,  so we also have 
\[
 \mathcal{J}_{\mm^*}= \varphi  \mathcal{J}_{{\vol}_\g}.
\]
Therefore $\varphi$ is a constant.

\end{proof}

\subsection{Measure rigidity:  geodesical convexity}

In the last theorem, we study the $\cd$ condition  on  manifolds with  Lipschitz boundary. Without loss of generality, we may restrict our study on an open set  $\Omega \subset M$ with  Lipschitz    boundary, which means that the boundary  $\partial \Omega$  can be written  locally as the graph of a Lipschitz continuous function on $\R^{n-1}$.

As we mentioned in the introduction, no matter how smooth the boundary is, we cannot definitely predict  that  the geodesics are $C^2$. Consider the complement of a disc in  the Euclidean plane.
A geodesic fails to have an acceleration only at those  points which  we call {\bf switch points}, where the geodesic switches from a boundary
segment to an interior segment or vice-versa.  
In addition, besides the switch points, boundary segments, and interior segments, one other kind of point is possible, an accumulation point of switch points, which we call  {\bf intermittent point}.
It is not difficult to   construct a geodesic  whose intermittent points form  a Cantor set with  positive measure. Unfortunately,  it is uncertain which assumptions on the boundary guarantee finite switching behavior. One known result (c.f. \cite{ABB-R}) is that domains in Euclidean plane with analytic boundary have  no intermittent point.  However, thanks to a  theorem proved by Alexander, Berg and Bishop (Theorem 1 \cite{ABB-R}), these intermittent points will not 
 bring us too much trouble in our problem.

\begin{theorem}[Measure rigidity:  $\cd$ condition]\label{th:convex}
Let $(M, {\rm g})$ be a complete  Riemannian manifold,  $\Omega \subset M$ be an open set with Lipschitz boundary. Let  $\d_\Omega$ be the intrinsic distance induced by the  Riemannian distance $\d_\g$ on $\overline{ \Omega}$, and $\mm$ be a reference measure with $\supp \mm=\overline \Omega$.  Assume that $\partial \Omega$ is $C^2$ out of a $\mathcal H^{n-1}$-negligible set, and  $(\overline{\Omega}, \d_\Omega, \mm)$  satisfies the  $\cd$ condition, then we have the following rigidity results.
\begin{itemize}
\item [1)]   $\overline \Omega$ is ${\rm g}$-geodesically convex,  this is to say,  any shortest path in  $(\overline { \Omega}, \d_\Omega)$ is a (unparameterized) geodesic  (segment)  in $(M, {\rm g})$;
\item [2)]   $\mm (\partial \Omega)=0$ and  $\mm=e^{-V}{\rm Vol}_{\rm g}$ for some  semi-convex, locally Lipschitz  function $V$ on $\Omega$;
\item [3)]  $(\overline{\Omega}, \d_\Omega, \mm)$  is a  $\rcd$ space.
\end{itemize}

In particular,   $(\overline{\Omega}, \d_\Omega, {\vol}_\g)$  is $\cd$  if and only if  $\overline \Omega$ is $\g$-geodesically convex and $\Ric \geq K$ on $\Omega$.
\end{theorem}
\begin{proof}

Since all the assertions are local,  without loss of generality, we  may assume that $\overline\Omega$ is compact and points in $\overline\Omega$ do not have cut-locus inside $\overline \Omega$.

First  of all, by Proposition \ref{lemma:rcddensity} and Proposition \ref{prop:differentiable} we know $\mm\restr{\mathop{{\Omega}}} \ll {\vol}_\g$,  ${\vol}_\g \ll \mm\restr{\mathop{{\Omega}}}$ and $\mm=e^{-V}{\rm Vol}_{\rm g}$ for some  semi-convex, locally Lipschitz  function $V$. In particular,  we have
\begin{equation}\label{eq0.1:thm}
\frac{\d \vol_\g}{\d \mm}\in L^\infty(\Omega, \mm),~~~ \frac{\d \mm}{\d \vol_\g} \in L^\infty(\Omega, \vol_\g).
\end{equation}

Given  $x, y \in {\Omega}$ and a parameter $\epsilon>0$ such that $B_\epsilon(x), B_\epsilon(y) \subset \Omega$.  
Firstly, consider  the $L^1$-optimal transportation on $(\overline{\Omega}, \d_\Omega)$ between  $\mu^\epsilon_0:= \frac 1{\mm\big(B_ \epsilon (x)\big)} \mm \restr{B_ \epsilon (x)}$ and $\mu^\epsilon_1:= \frac 1{\mm\big(B_ \epsilon (y)\big)} \mm \restr{B_\epsilon (y)}$. 
Let $(\mu^\epsilon_t)_t$  be a  geodesic from $\mu^\epsilon_0$ to $\mu^\epsilon_1$ in $L^1$-Wasserstein space $\mathcal{W}_1(\overline{\Omega}, \d_\Omega)$. Denote  by $\Pi^\epsilon$ its lifting in $ {\mathcal P}\big({\rm Geod}(\overline\Omega, \d_\Omega)\big)$  satisfying  $(e_t)_\sharp \Pi^\epsilon =\mu_t^\epsilon$.  By $L^1$-optimal transport theory,   there exists a  Kantorovich potential $\varphi$ associated with such optimal transportation, which  is a  1-Lipschitz function.   Let  $ \Gamma^\varphi$  be   the subset  of ${\rm C}\big([0,1]; (\overline \Omega, \d_\Omega)\big)$ containing all the (parametrized)  trajectories   of the gradient flow of  $\varphi$.   It  is known that $\Pi^\epsilon(\Gamma^\varphi)=1$.

For $0<\delta \ll \frac 12$ small enough,  $(\mu^\epsilon_t)_{t\in [0, \delta]}$ and $(\mu^\epsilon_t)_{t\in [1- \delta, 1]}$ are  also $L^1$-Wasserstein geodesics (segments) in $\mathcal{W}_1(\overline{\Omega}, \d_\g)$. 
By needle decomposition via $L^1$-optimal transport (c.f. Theorem 3.8, Theorem 5.1 \cite{CM-ISO}),  there is $\Gamma \subset \Gamma^\varphi$ such that $\Pi^\epsilon(\Gamma^\varphi \setminus \Gamma)=0$ and  $(\gamma_t)_{t\in [0,\delta] \cup [1-\delta,1]}, \gamma \in \Gamma$ are pairwisely disjoint.  In addition,  the measure  ${\vol}_\g \restr{B_\epsilon(x)}$ (and similarly ${\vol}_\g \restr{B_\epsilon(y)}$) has a decomposition
\begin{equation}\label{eq1:thm}
{\vol}_\g \restr{B_\epsilon(x)}=\int_\Q \mm_q \, \d \mathfrak{q}
\end{equation}
where $\Q$ can be represented  locally as  a level set of $\varphi$, and $(\mm_q)_q$ support on disjoint geodesic segments $(X_q)_q$. Furthermore,  $\mm_q \ll \mathcal H^1$ and $h_q= \frac{\d \mm_q}{\d \mathcal H^1}$ is a ${\rm CD}(k, n)$ density for $\mathfrak{q}$-a.e. $q$.,  this means that  $(X_q, |\cdot|, \mm_q)$ satisfies  ${\rm BE}(k, n)$ condition in the sense of Bakry-\'Emery. 
Thus for any $q$, $h_q$ is Lipschitz  and  it can not vanish at the interior points  of $X_q$.

Next  we  will construct a $L^2$-optimal transportation based on  $\Pi^\epsilon$ and $\Gamma$. 

 Denote the level set $\{ \varphi=T\}$ by $\varphi_T$.  For any $z \in B_\epsilon(x)$,  there exist a $\gamma^z \in \Gamma$  such that $z \in  \gamma^z$, and  a unique  $T_z:=\varphi(z)$ such that $z = \varphi_{T_z} \cap \gamma^z$.  In addition, by Fubini's theorem, there exists $T_0$ such that $B^*:=\{z:  \gamma^z \cap \varphi_{T_z-T_0}\in B_\epsilon(y)\} \cap B_\epsilon(x)$ has positive ${\vol}_\g$-volume.
It can be seen that $\mathrm{Cpl}:=\{(z_1, z_2): z_1 \in B^*, z_2 \in \gamma^{z_1} \cap \varphi_{T_{z_1}-T_0} \} \subset   B_\epsilon(x) \times B_\epsilon(y)$  is still a $L^1$-optimal transport coupling.
Furthermore,  we have
\begin{eqnarray}\label{eq2:thm}
&&\Big (\varphi(y_1)-\varphi(y_0)\Big )\Big (\varphi(x_1)-\varphi(x_0)\Big )\\
&=&\Big (\big (\varphi(x_1)-T_0 \big)-\big ( \varphi(x_0)-T_0 \big) \Big ) \Big (\varphi(x_1)-\varphi (x_0)\Big )\\
&=& \Big (\varphi(x_1)-\varphi (x_0)\Big )^2 \geq 0
\end{eqnarray}
for any $(x_0, y_0), (x_1, y_1) \in \mathrm{Cpl}$.
By Lemma 4.6 in  \cite{Cavalletti-D},  we know $\mathrm{Cpl}$ is  $\d^2_\Omega$-cyclically monotone, so that it is  also  a  $L^2$-optimal transport coupling (c.f. Theorem 2.13 \cite{AG-U}).  From the construction of $\mathrm{Cpl}$ we know
$\big (\mathrm{Cpl} \big)_{z_1}=B^*$ has  positive $\mm$-measure,  and by measure decomposition \eqref{eq1:thm} and the regularity of ${\rm CD}(k, n)$ densities, we  also have $\mm\Big (\big ( \mathrm{Cpl}\big)_{z_2} \Big )>0$.
Then  by renormalization we obtain a curve, still denote it  by $(\mu^\epsilon_t)$, which is a $L^1$-Wasserstein geodesic, as well as a $L^2$-Wasserstein geodesic.    From the construction above, we can  see that  both $\mu^\epsilon_0, \mu^\epsilon_1$ have bounded  $\mm$-densities.

  To prove the geodesical convexity of $\overline {\Omega}$,  we just need to show that  $\Pi^\epsilon \big({\rm Geod}(\overline \Omega, \d_\Omega)\setminus {\rm Geod}(M, {\rm g}) \big )=0$, then letting $\epsilon \to 0$  we know that $x$ and  $y$ are connected by a geodesic in $ (M, {\rm g})$.

 Let   $\mathrm {R}$ be  the set of $C^2$-regular points of $\partial \Omega$.  By  assumption, $\mathcal{H}^{n-1}(\partial \Omega \setminus \mathrm R)=0$.  It can be seen that $ {\rm Geod}(\overline \Omega, \d_\Omega)$ has a decomposition  $  {\rm Geod}(\overline \Omega, \d_\Omega)=\Gamma^1 \cup \Gamma^2  \cup \Gamma^3$, where
\begin{itemize}
\item [a)] $\Gamma^1=\Big \{\gamma:  \mathcal{H}^1(\gamma \cap \mathrm R) >0 \Big \}$;
\item [b)] $\Gamma^2=\Big \{\gamma: \mathcal{H}^1(\gamma \cap \mathrm R) =0$, $\gamma \cap \partial \Omega \subset \mathrm {R} \Big \}$;
\item [c)] $\Gamma^3=\Big \{\gamma: \mathcal{H}^1(\gamma \cap \mathrm R) =0$, $\gamma \cap \partial \Omega \setminus \mathrm {R} \neq \varnothing\Big \}$.
\end{itemize}
We will   prove $\Pi^\epsilon  (\Gamma^i )=0$ for $ i=1, 3$,  and $\Gamma^2 \subset  {\rm Geod}(M, {\rm g})$ in the following three steps.

\bigskip

{\bf Step 1}:  $\Pi^\epsilon(\Gamma^1)=0$.

   By  Proposition \ref{lemma:rcddensity} and Proposition \ref{prop:differentiable},  there exists a locally Lipschitz and  semi-convex function $V$ such that
 $$
 \mm=e^{-V}{\rm Vol}_{\rm g} \restr{\Omega}+\mm\restr{\partial \Omega\setminus \mathrm{R}}+\mm\restr{ \mathrm{R}}.
 $$  
 In particular, $\mm \restr{\Omega} < C_0 {\rm Vol}_{\rm g} \restr{\Omega}$ for some $C_0>0$.
 
\bigskip

{\bf Claim}:  $\mm(\mathrm{R})=0$, therefore  $\mm=e^{-V}{\rm Vol}_{\rm g} \restr{\Omega}+\mm\restr{\partial \Omega\setminus \mathrm{R}}$.

Assume that  $\partial \Omega$ is locally represented  as the graph  of a bi-Lipschitz function  $\phi$ on $U \subset \R^{n-1}$, and 
\[
\frac 1L |ab|<\d_\g(x, y) < L |ab|~~~\forall a, b\in U,~x=(a, \phi(a)), y=(b, \phi(b))
\] for some $L>1$.  For any  $a \in U$ with $(a, \phi(a) )\in  \mathrm {R}$, there is a unique tangent plane $\d \phi(a)$ at this point and  
 \[
 \lmt{r}{0} \mathop{\sup }_{b\in U, |b-a|<r} \frac {\big |\phi(b)-\phi(a)-\d \phi(a)(b-a) \big |}{|b-a|}=0.
 \]
In particular,   there exists  a unique  inward  (unit) normal vector field  $ \mathrm R \ni x \mapsto {\rm N}_x $   such that  ${\rm N}_x \perp \d \phi(x)$.  Furthermore, for any $x \in \mathrm R$, 
there is $\delta(x)>0$ such that  for any $r\leq \delta(x)$, all $\d_\Omega$-geodesics from $B_{ {r}} \big ( \exp_x( {3r}  {\rm N_x})\big)$ to $x$ are $ \g$-geodesics.

  Assume  by contradiction that $\mm(\mathrm{R}) \neq 0$.  By Lusin's theorem there exists ${\mathrm R}^* \subset \mathrm R$  with  $\mm({\mathrm R}^*)>0$ and a constant $r_0>0$,  such that $\delta(x) \geq r_0$ and  the  map $x \to {\rm N}_x$ is continuous on $\mathrm R^*$.
 
 Let $x \in \mathrm R^* \cap \supp \mm \restr{\mathrm R^*}$.
   There is a neighbourhood $U_x \subset \partial \Omega$ of $x$,  such that  all the $\d_\Omega$-geodesics connecting $U_x \cap  \mathrm R^*$ and $B_{\frac{r_0}{2}}\big( \exp_x({3r_0} {\rm N_x})\big)$ are $ \g$-geodesics.   By Lemma \ref{lemma:sigular},   we get the contradiction.  Therefore  $\mm(\mathrm{R})=0$ and we prove the claim.

   \bigskip
   
Assume by contradiction that $\Pi^\epsilon(\Gamma^1)>0$.  By Fubini's theorem we know
\[
(\Pi^\epsilon \times L^1)\Big (\big \{(\gamma, t): \gamma\in \Gamma_1, t\in [0,1], \gamma_t \in  \mathrm R \big \} \Big ) >0,
\]
and there is  $t_0 \in [0,1]$ such that 
\[
\Pi^\epsilon \Big (\big \{\gamma: \gamma\in \Gamma^1,  \gamma_{t_0} \in \mathrm R \big \}\Big) >0,
\]
so $\mu_{t_0}^\epsilon(\mathrm R)>0$, which contradicts to the facts that $\mu^\epsilon_{t_0} \ll \mm$ and  $\mm\restr{\mathrm R}=0$.
Therefore $\Pi^\epsilon(\Gamma^1)=0$.

\bigskip

{\bf Step   2}:  $\Gamma^2 \subset  {\rm Geod}(M, {\rm g})$.

 Let $\gamma \in \Gamma^2$. For any   $t \in [0,1]$ with  $\gamma_t \in   \Omega$,  it is known that  $\ddot{\gamma}_t=0$.  For any   $t \in [0,1]$ with  $\gamma_t \in \partial  \Omega$.  From the definition  of $\Gamma^2$ we know $\gamma_t \in \mathrm R$,  so $\dot{\gamma}_t$ exists. 

Since $\gamma \cap \partial \Omega$ is closed and $\mathcal H^{1}$-negligible,  any $\gamma_t \in \mathrm R$ is either an isolate point or an intermittent point. For isolate points, by elementary calculus we know  $\ddot{\gamma}_t$ exists and equals to $0$. For  intermittent points, by Theorem 1  in \cite{ABB-R} we also know  $\ddot{\gamma}_t$ exists and equals to $0$.

 So  $\ddot{\gamma}_t \equiv 0$ and $\gamma \in  {\rm Geod}(M, {\rm g}) $ which is the thesis.

\bigskip

{\bf Step 3}:    $\Pi^\epsilon(\Gamma^3)=0$.

By definition and  $\Gamma^2 \subset  {\rm Geod}(M, {\rm g})$ which is proved in the last step,   for any $\gamma \in \Pi^\epsilon(\Gamma^3)$,  there is a point  $b(\gamma) \in \gamma \cap \partial \Omega \setminus \mathrm {R}$ such that the segment from $\gamma_0$ to $b(\gamma)$ is a Riemannian geodesic (segment).

Recall the decomposition  ${\vol}_\g \restr{B_\epsilon(x)}=\int_\Q \mm_q \, \d \mathfrak{q}$ in \eqref{eq1:thm} and keep the notations thereafter.   For any $z \in B_\epsilon(x)$,  there a longest $\gamma^z \in \Gamma^3$  such that $z \in  \gamma^z$, and  a unique  $T_z =\varphi(z) \in\R$ such that $z = \varphi_{T_z}$.
In addition,  for such $\gamma^z$, there is a unique $q^z \in \Q$ such that $\supp \mm_{q^z} \subset \gamma^z$.

Assume by contradiction that  $\Pi^\epsilon(\Gamma^3)>0$.  Denote $\Q^T$ by
\[
\Q^T:=\Big \{ q^z: z\in B_\epsilon(x), \gamma^z\in \Gamma^3, \varphi(\gamma^z_0)-\varphi(z)>\frac \epsilon 8~\text{and}~\varphi_ {T_z-T} \cap \gamma^z=b(\gamma^z)   \Big \}.
\]
Then by Fubini's theorem,  there is $T_1>0$ such that 
$\mathfrak{q}\Big (  \Q^{T_1} \Big )>0$.


For  $\sigma \in (0, \delta)$,   we define  couplings $\mathrm{Cpl}^{1,2}_\sigma, \mathrm{Cpl}^{1,3}_\sigma \subset M \times M$ by
\[
\mathrm{Cpl}^{1,2}_\sigma:=\Big \{(z_1, z_2):  z_1 \in \gamma \cap \varphi_{T_{b(\gamma)}+T_1-s\sigma},     z_2 \in \gamma \cap \varphi_{T_{b(\gamma)}+(1-s)\sigma},     \gamma \in \Gamma^{3}, s\in [0,1]  \Big \}
\]
and
\[
\mathrm{Cpl}^{1,3}_\sigma:=\Big \{(z_1, z_3):  z_1 \in \gamma \cap \varphi_{T_{b(\gamma)}+T_1-s\sigma},     z_3 \in \gamma \cap \varphi_{T_{a(\gamma)}-s\sigma-T_0},     \gamma \in \Gamma^{3}, s\in [0,1]  \Big \}.
\]

By the choice of $T_1$ and the regularity of conditional measures $\mm_q$,  we can  see that 
$\mm\big((\mathrm{Cpl}^{1,2}_\sigma)_{z_1}\big),  \mm\big((\mathrm{Cpl}^{1,3}_\sigma)_{z_3}\big)>0$. More precisely,  we have
\begin{equation}\label{eq3:thm}
\vol\Big((\mathrm{Cpl}^{1,2}_\sigma)_{z_1}\Big)\geq \int_{\Q^{T_1}} \mm_q \Big(\big\{\varphi_{T_{b(\gamma)}+T_1-s\sigma}\cap \gamma: \gamma \in \Gamma^3, s\in [0, 1]\big\}\Big) \, \d \mathfrak{q}=O(\sigma)>0,
\end{equation}
and similarly
\begin{equation}\label{eq3.1:thm}
\vol\Big((\mathrm{Cpl}^{1,2}_\sigma)_{z_1}\Big)\gtrsim \sigma.
\end{equation}

From the construction,  we can see that these couplings are $\d_\Omega$-cyclically monotone, as well as  $\d^2_\Omega$-cyclically monotone (c.f.  \eqref{eq2:thm} and Lemma 4.6  \cite{Cavalletti-D}). Therefore, they are both $L^1$-optimal and $L^2$-optimal (c.f. Theorem 2.13 \cite{AG-U}).   By renormalization and reparameterization,  we can   find  a Wasserstein geodesic $(\nu^\sigma_t)$  in $\mathcal{W}_2(\overline{\Omega}, \d_\Omega) \cap \mathcal{W}_1(\overline{\Omega}, \d_\Omega)$,    such that $\mathrm{Cpl}^{1,2}_\sigma$ is the optimal coupling for $(\nu^\sigma_0, \nu^\sigma_\frac12)$ and $\mathrm{Cpl}^{1,3}_\sigma$ is the optimal coupling for $(\nu^\sigma_0, \nu^\sigma_1)$ and
  \begin{itemize}
\item [1)]   $(\nu^\sigma_t)_{t\in [0,\frac12]}$ is  a geodesic segment  in Wasserstein space $\mathcal{W}_2 (\overline{\Omega}, \d_\g)$;
\item [2)]   Given $\delta \in (0, \frac12)$,   $(\nu^\sigma_t)_ {t\in [0,\frac12-\delta] \cup \{1\}}$ have uniformly  bounded  $\mm$-densities;
\item [3)]   $\mm(\supp \nu^\sigma_0) \gtrsim \sigma$ and   $\mm(\supp \nu^\sigma_1)\gtrsim \sigma$ (by  \eqref{eq3:thm}, \eqref{eq3.1:thm} and \eqref{eq0.1:thm}).
\end{itemize}

Moreover, since  $\mathcal{H}^{n-1}(\partial \Omega \setminus \mathrm {R})=0$,  by Rauch's comparison theorem we know ${\vol}_\g(\supp \nu^\sigma_\frac12) \lesssim  \sigma^n$, so that  $\mm(\supp \nu^\sigma_\frac12) \lesssim  \sigma^n$.

Since  $(\overline{\Omega}, \d_\Omega, \mm)$ is $\cd$,   by Lemma 3.1  \cite{Rajala12} there exists a $L^2$-Wasserstein geodesic   $(\bar{\nu}^\sigma_t)_{t\in [0,1]} \subset \mathcal{W}_2(\overline{\Omega}, \d_\Omega)$ with uniformly  bounded densities,   connecting  $\nu^\sigma_\frac14$ and  $\nu^\sigma_1$ such that 
\[
\mm(\supp {\bar \nu}^\sigma_t) \gtrsim  \min \Big  \{\mm(\supp \nu^\sigma_\frac14), \mm(\supp \nu^\sigma_1) \Big \},~~~t\in [0,1].
\]
It is known that  Riemannian manifolds are essentially non-branching, hence  ${\nu}^\sigma_s \in (\bar{\nu}^\sigma_t)$ for all $ {s\in [\frac14, \frac12]}$. In particular, there exists $t_1 \in (0, 1)$ such that  
 ${\nu}^\sigma_{\frac 12} =\bar{\nu}^\sigma_{t_1}$.
Therefore we have
\[
\sigma^n \gtrsim \mm(\supp { \nu}^\sigma_\frac12)=\mm(\supp {\bar \nu}^\sigma_{t_1}) \gtrsim O(\sigma)
\]  which  is the 
contradiction. 
Therefore  $\Pi^\epsilon(\Gamma^3)=0$.

\bigskip

In conclusion, we have proved that   $(\overline{\Omega}, \d_\Omega, \mm)$   is  $(M, \g)$-geodesically convex. By Lemma \ref{lemma:sigular} we have  $\mm\restr{\partial \Omega} = 0$.
 Therefore $\mm=e^{-V}{\rm Vol}_{\rm g} \restr{\Omega}$ for some Lipschitz function $V$.  In particular,   $(\overline{\Omega}, \d_\Omega, \mm)$ is infinitesimally Hilbertian and  it satisfies $\rcd$ condition.

\end{proof}

\def\cprime{$'$}

\end{document}